\documentclass{amsart}
\usepackage[utf8]{inputenc}
\usepackage[english]{babel}
\usepackage{verbatim}
\usepackage{graphicx}
\usepackage{amsmath,amsthm,amssymb}
\usepackage{afterpage} 
\usepackage{tcolorbox} 
\usepackage{float}
\usepackage{subfig}
\usepackage{enumerate}
\usepackage{todonotes}

\newtheorem{Theorem}{Theorem}
\newtheorem{Corollary}[Theorem]{Corollary}
\newtheorem{Proposition}[Theorem]{Proposition}
\newtheorem{Lemma}[Theorem]{Lemma}
\newtheorem{Definition}[Theorem]{Definition}
\newtheorem{Remark}[Theorem]{Remark}
\newtheorem{Property}[Theorem]{Property}

\newtheorem{Example}[Theorem]{Example}

\begin{document}

\title{Variational principles and applications to symmetric PDEs}

\author[J. Falc\'{o}]{Javier Falc\'{o}}
\address[Javier Falc\'{o}]{Departamento de An\'{a}lisis Matem\'{a}tico,
	Universidad de Valencia, Doctor Moliner 50, 46100 Burjasot (Valencia), Spain} \email{francisco.j.falco@uv.es}
\author[D. Isert]{Daniel Isert}
\address[Daniel Isert]{Departamento de An\'{a}lisis Matem\'{a}tico,
	Universidad de Valencia, Doctor Moliner 50, 46100 Burjasot (Valencia), Spain}
\email{daniel.isert@uv.es}

\thanks{Both authors were supported by the Spanish Ministry of Science, Innovation and Universities and the Spanish State Research Agency, Project PID2025-168101NB-I00 funded by MICIU/AEI/10.13039/501100011033. The second author was also supported by the Spanish Ministry of Science, Innovation and Universities and the Spanish State Research Agency, Project PID2025-171414NA-I00 funded by MICIU/AEI/10.13039/501100011033)}



\begin{abstract}
In this paper, we explore various equivalences of Ekeland's variational principle within the framework of group-invariant mappings. We introduce and analyze several key theorems, including the Drop theorem, the Petal theorem, the Caristi-Kirk fixed-point theorem, and the Takahashi's theorem, all of them within this context. Moreover, we extend the classical Drop theorem and Petal theorem to a more generalized setting. We also demonstrate the practical significance of these findings through numerous applications to diverse areas of mathematics. In particular, in the context of partial differential equations, we explore their implications on the solution of the Plateau problem, and in control theory. We also extend the classical Pontryagin maximum principle.
\end{abstract}

\maketitle


\section{Introduction}

Since the publication of the Ekeland's variational principle, \cite{Ekeland}, this result has proved to be one of the most powerful tools in analysis due to its wide-range of applications to functional analysis, fixed point theory, and partial differential equations, among others. Recently, the theory of group invariant mappings has gained a lot of strength inside the field of functional analysis, in its most general way, see for instance \cite{ArFaMa}, \cite{DaFaJu}, \cite{DaFaJuRo}, and \cite{FaIs1}. Also, there is a new theory of group equivariant operators that is being studied in the theory of geometry of Banach spaces, see for instance \cite{DaRiRa}, \cite{DaDoJuRa}, and \cite{Doucha}.

In 2024 we proved a group invariant version of the Ekeland's variational principle, see \cite{FaIs}, and studied several applications of this result. The purpose of this paper is to deepen the study of the Ekeland's variational principle by giving some equivalences of this result, such as the Drop theorem, the Petal theorem, the Caristi-Kirk fixed point theorem, and the Takahashi's theorem, all of them in their group invariant version. Furthermore, in our study of the Drop and the Petal theorem, by using group invariant techniques, we obtain generalizations of the classical ones. In this latter results, instead of searching for a group invariant point, we search for group invariant sets on the solution. 

As we mentioned earlier, one of the main applications of Ekeland's variational principle is to solve problems in partial differential equations (PDE). Continuing in this line, we will present here some new applications of these results such as solving symmetric PDEs. 

We will begin this manuscript by presenting all the notation and definitions we will require throughout the paper, as well as some useful results that will be used later on. Then, we will move to the core of the paper, where we study the equivalences of the group invariant Ekeland's variational principle. In particular we present the group invariant version of the Drop theorem, the Petal theorem, the Caristi-Kirk fixed point theorem, and the Takahashi's theorem. To conclude we will present applications in the fields of geometric analysis and PDEs.

\section{Preliminaries}

Let us begin by establishing some notation. All normed and Banach spaces in this paper are real unless explicitly stated otherwise. We will denote a normed space by $(X, \Vert \cdot \Vert)$, and the space of linear and continuous mappings from $X$ to $X$ by $\mathcal{L}(X)$. Hereinafter, $G \subseteq \mathcal{L}(X)$ is reserved to denote a topological group of invertible bounded linear isometries with the relative topology of $\mathcal{L}(X)$. Also, the letter $g$ is reserved to denote an element of the group $G$.

To start, let us present the three notions of $G$-invariance. Let $X$, $Y$ be two normed spaces and let $G$ be a topological group. We say that:
\begin{enumerate}
\item A point $x \in X$ is $G$-invariant, or invariant under the action of $G$, if $g(x) = x$ for all $g \in G$.

\item A set $K \subset X$ is $G$-invariant if for every $g \in G$, $g(K) = K$.

\item A mapping $f \colon X \to Y$ is $G$-invariant if for every $x \in X$ and every $g \in G$ we have that
\[
f(g(x)) = f(x).
\]
\end{enumerate}

As usual, we will denote by $X_{G}$, $B_{X}$, $S_{X}$ the set of all group invariant points of $X$, the open unit ball of $X$, and the unit sphere of $X$ respectively. We will always assume that the norm of the space $X$, $\Vert \cdot \Vert$, is $G$-invariant.

A fundamental definition of this paper is that of $G$-symmetrization point. Let us recall here what we understand by $G$-symmetric point. Let $X$ be a Banach space and $G$ be a compact topological group acting on $X$. If $x \in X$ we define the symmetrization point of $x$ with respect to $G$, or the $G$-symmetrization point, as
\[
\overline{x} = \int_{G}g(x)d\mu(g),
\] 
where $\mu$ is the normalized Haar measure and the integral is the Bochner integral. Recall that, since the normalized Haar measure of a compact group is bi-invariant and has total mass one, the point $\overline{x}$ is $G$-invariant and satisfies $\overline{g(x)} = \overline{x}$ for every $g \in G$; in particular $\overline{\overline{x}} = \overline{x}$, and $\overline{x} = x$ whenever $x$ is $G$-invariant. For a normed space which is not complete, we take this integral in its completion; stability of a subset under symmetrization requires these averages to belong to the subset itself.

In \cite[Theorem 12]{FaIs} we proved Ekeland's $G$-invariant version under the assumption of convexity with respect to the group $G$; this assumption cannot be omitted in general, as shown in \cite[Example 13]{FaIs}. Recall that if $X$ is a Banach space and $G$ is a compact topological group acting on $X$, a function $\varphi \colon X \to \mathbb{R}$ is convex with respect to $G$ given that 
\[
\varphi\left(\int_{G}g(x) d\mu(g)\right) \leq \int_{G}\varphi(g(x))d\mu(g) \quad \forall \, x \in X.
\]
And we say that $\varphi$ is linear with respect to $G$ given that
\[
\varphi\left(\int_{G}g(x) d\mu(g)\right) = \int_{G}\varphi(g(x))d\mu(g) \quad \forall \, x \in X.
\]
In this paper we will also work with metric vectorial spaces that we will denote always as $(M, d)$. In this case we will also consider $G \subseteq \mathcal{L}(M)$. In this metric setting we assume that a symmetrization map $x \mapsto \overline{x}$ is defined in $M$, takes values in $M_{G}$, fixes every point of $M_{G}$, and satisfies $\overline{g(x)} = \overline{x}$. Whenever the Haar integral exists in $M$, it defines this map; convexity and linearity with respect to $G$ are understood using this map and the same scalar integral inequalities. We say that a metric $d \colon M \times M \to [0, +\infty[$ is $G$-invariant if
\[
d(g(x), g(y)) = d(x, y) \quad \forall \, x, y \in M, \quad \forall \, g \in G
\]
Note that, if the distance is translation invariant, by linearity of the group, this is equivalent to assume that $d(x,0)=d(g(x),0)$ for all $x\in M$ and all $g\in G$. As in the case of normed spaces, we will always assume that the distance of the metric spaces is $G$-invariant.

We say that a metric $d \colon M \times M \to [0, +\infty[$ is convex with respect to $G$ if
\[
d(\overline{x}, \overline{y}) \leq \int_{G}d(g(x), g(y))d\mu(g).
\]
\begin{Remark}\label{Remark distancia convexa}
    Observe that if $d$ is induced by a norm, then it is convex with respect to $G$.  Both integrals below are taken with respect to the same normalized Haar measure $\mu$ on $G$. The integral of $g \mapsto g(x-y)$ is a Bochner integral, taken in the Banach completion if necessary, whereas the integral of $g \mapsto \Vert g(x-y)\Vert$ is a scalar Lebesgue integral. By linearity of the Bochner integral,
    \[
    d(\overline{x}, \overline{y}) = \left\Vert \int_{G}g(x-y)d\mu(g)\right\Vert \leq \int_{G}\Vert g(x-y)\Vert d\mu(g) = d(x,y).
    \]
\end{Remark}

In section \ref{Equivalència de la gota} we are going to work with the petal of two points, and the drop of a set. Let us present here this two objects, and study some interesting $G$-invariant properties related to them. Let $(M, d)$ be a metric space, $a, b \in M$ and $\gamma > 0$. The petal of $a$ and $b$ is defined as
\[
P_{\gamma}(a, b) = \left\{x \in M ~ | ~ \gamma d(a, x) + d(b, x) \leq d(a, b)\right\}.
\]
Let $(X, \Vert \cdot \Vert)$ be a normed vector space, $a \in X$ and $B \subset X$ be a nonempty convex subset. The drop of $a$ and $B$ is defined as
\[
D(a, B) = \left\{a + t(b - a) ~ | ~ t \in [0, 1], \, b \in B\right\} = \text{conv}(a, B),
\]
where $\text{conv}(a, B)$ denotes the convex hull of $a$ and $B$. When $B = \left\{b\right\}$, we write $D(a,b) = D(a,\left\{b\right\}) = [a,b]$.

Let us start by showing the following properties of the drop and the petal when we consider a group $G$.

Our first result is a direct consequence of the $G$-invariance of the distance, but we include its proof here for the sake of completeness.
\begin{Proposition}\label{Propietat petal}
    Let $(M, d)$ be a metric vectorial space and $G \subseteq \mathcal{L}(M)$ be a topological group of isometries acting on $M$. Let $\gamma > 0$, then
    \[
    g(P_{\gamma}(a, b)) = P_{\gamma}(g(a), g(b)) \quad \forall \, a, b \in M, \hbox{ and } \forall \, g \in G.
    \]
\end{Proposition}
\begin{proof}
Let $x \in P_{\gamma}(a, b)$. By definition
\[
d(a, b) \geq \gamma d(a, x) + d(x, b).
\]
Since every $g \in G$ is invertible, this is equivalent to
\[
d(g^{-1}(g(a)), g^{-1}(g(b))) \geq \gamma d(g^{-1}(g(a)), g^{-1}(g(x))) + d(g^{-1}(g(x)), g^{-1}(g(b))).
\]
And by $G$-invariance of $d$, this is the same as
\[
d(g(a), g(b)) \geq \gamma d(g(a), g(x)) + d(g(x), g(b)) \quad \forall \, g \in G.
\]
so $g(x) \in P_{\gamma}(g(a), g(b))$. This shows that $g(P_{\gamma}(a, b)) \subseteq P_{\gamma}(g(a), g(b))$.

Take now $x \in P_{\gamma}(g(a), g(b))$, we want to see that $g^{-1}(x) \in P_{\gamma}(a, b)$. Notice that
\[
d(a, b) \geq \gamma d(a, g^{-1}(x)) + d(b, g^{-1}(x)),
\]
if, and only if,
\[
d(g(a), g(b)) \geq \gamma d(g(a), g(g^{-1}(x))) + d(g(b), g(g^{-1}(x))),
\]
by the $G$-invariance of $d$. But this last inequality reads as follows:
\[
d(g(a), g(b)) \geq \gamma d(g(a), x) + d(g(b), x).
\]
This shows the other inclusion and concludes the proof.
\end{proof}

The next result is a direct consequence of the linearity of the elements of the group.
\begin{Proposition}
    Let $(X, \Vert \cdot \Vert)$ be a normed vector space, $G \subseteq \mathcal{L}(X)$ be a topological group of isometries acting on $X$, $a \in X$ and $B \subset X$ be a nonempty convex subset. Then:
    \[
    g(D(a, B)) = D(g(a), g(B)).
    \]
\end{Proposition}
    

The following $G$-invariant properties answer some of the natural questions that rise in our scenario of $G$-invariant mappings.
\begin{Property}
    Let $(M, d)$ be a metric vectorial space and $G \subseteq \mathcal{L}(M)$ be a topological group of isometries acting on $M$. Then, $a, b \in M$ are $G$-invariant if and only if $P_{\gamma}(a, b)$ is $G$-invariant for every $\gamma > 0$.
\end{Property}
\begin{proof}
The direct implication is clear. For the other implication, observe that if $a \in P_{\gamma}(a, b)$, then for all $\gamma > 0$, $g(a) \in P_{\gamma}(a, b)$ by $G$-invariance of $P_{\gamma}(a, b)$. Therefore
\[
0 \leq d(g(a), a) \leq \frac{1}{\gamma}(d(a, b) - d(g(a), b)).
\]
Letting $\gamma \to \infty$, we have that $a = g(a)$ for all $g \in G$. Then, $a$ is $G$-invariant. We may assume that $a \neq b$, since otherwise $b = a$ is $G$-invariant and there is nothing left to prove. Observe now that, since $d(a, b) > 0$, $b \in P_{\gamma}(a, b)$ if, and only if, $\gamma \leq 1$. Indeed,
\[
\gamma d(a, b) = \gamma d(a, b) + d(b, b) \leq d(a, b).
\]
In order to show now that $b$ is $G$-invariant, take $\gamma = 1$ and suppose, by contradiction, that $g(b) \neq b$. Then $g(b) \in g(P_{1}(a, b)) = P_{1}(a, b)$ since $P_{1}(a, b)$ is $G$-invariant. Therefore
\begin{equation}\label{EqPetal 1}
d(g(b), a) + d(g(b), b) \leq d(a, b).
\end{equation}
On the other hand, by Proposition \ref{Propietat petal} and using that $a$ is $G$-invariant, we know that $g(P_{1}(a, b)) = P_{1}(a, g(b))$. Thus $b \in P_{1}(a, b) = P_{1}(a, g(b))$, then
\begin{equation}\label{EqPetal 2}
d(b, a) + d(b, g(b)) \leq d(a, g(b)).
\end{equation}
Putting together equations \eqref{EqPetal 1} and \eqref{EqPetal 2}, we obtain that
\[
d(g(b), a) + d(g(b), b) \leq d(a, g(b)) - d(b, g(b)),
\]
that is, $2d(b, g(b)) \leq 0$, a contradiction with the assumption that $g(b) \neq b$. Then $b$ is $G$-invariant.
\end{proof}

The proof of the following result is a direct consequence of the linearity of the elements of the group.
\begin{Proposition}
    Let $(X, \Vert \cdot \Vert)$ be a normed space, $G \subseteq \mathcal{L}(X)$ be a topological group of isometries acting on $X$, and $A, B \subseteq X$ be two $G$-invariant subsets. Then $\text{conv}(A, B)$ is $G$-invariant.
\end{Proposition}

As a consequence of this result, we obtain the following.
\begin{Corollary}
    Let $(X, \Vert \cdot \Vert)$ be a normed space and $G \subseteq \mathcal{L}(X)$ be a topological group of isometries acting on $X$. Assume that $x_{0} \in X$, $B \subset X$ is a nonempty convex subset, and that both are $G$-invariant. Then $D(x_{0}, B)$ is $G$-invariant.
\end{Corollary}

We are going to show now that the previous hypotheses are natural sufficient conditions and that, if they are removed, $G$-invariance of the resulting set can no longer be guaranteed in general. Through the following examples we will be working with the group $G = \left\{Id, \sigma\right\} \subseteq \mathcal{L}(\mathbb{R}^{2})$, where $\sigma \colon \mathbb{R}^{2} \to \mathbb{R}^{2}$ is defined as $\sigma(x, y) = (y, x)$.
\begin{Example}
In $\mathbb{R}^{2}$ take $b = (0, 0)$, which is clearly $G$-invariant, and $a = (2,0)$, which is not $G$-invariant. Then, the petal $P_{\gamma}(a, b)$ is not $G$-invariant. Also, take $B = B((0, 3), 1)$ the open ball centered at $(0, 3)$ with radius $1$, which is clearly not $G$-invariant. Then, the drop $D(b, B)$ is not $G$-invariant.
\end{Example}

However if we consider a non $G$-invariant point, $a$, we cannot say anything about the drop.
\begin{Example}
In $\mathbb{R}^{2}$ suppose $a = (0, 1)$ which is not $G$-invariant, and take $B = B((0, 0), 2)$ the ball centered at $(0, 0)$ with radius $2$, which is clearly $G$-invariant, by $G$-invariance of the norm. Then $D(a, B) = B = B((0, 0), 2)$ which is $G$-invariant as before. Now, suppose $a = (5, 0)$ which is not $G$-invariant, and take $B = B((0, 0), 2)$ the ball centered at $(0, 0)$ with radius $2$, which is clearly $G$-invariant. Since $\sigma(a) = (0, 5)$, we can see that $a \in D(a, B)$, but $a \notin D(\sigma(a), B)$. So, $D(a, B)$ is not $G$-invariant.
\end{Example}

Now, we present here the definition of a set being stable under symetrizations, which will be used through the manuscript.
\begin{Definition}
Let $X$ be a normed space and $G \subseteq \mathcal{L}(X)$ a topological group acting on $X$. Let $C \subseteq X$ be a subset, we say that $C$ is stable under symmetrization given that $\overline{x} \in C$ for every $x \in C$. 
\end{Definition}
Observe that by \cite[Proposition 2.2]{DaFaJu} we know that $G$-invariant convex subsets are stable under symmetrization. Through the paper, we will need to perform some operations on group invariant sets. These operations are compatible with the $G$-invariance of the sets. We would like to conclude this section, by recalling some of these results. The proofs of these results are straightforward.
\begin{Property}
    Let $E$ be a normed vectorial space or a metric vectorial space, and let $G \subseteq \mathcal{L}(E)$ be a topological group of isometries acting on $E$. If $A, B \subseteq E$ are two $G$-invariant subsets, then $A \cap B$ is again $G$-invariant and $(A \cap B)_{G} = A_{G} \cap B_{G}$.
\end{Property}

\begin{Property}
    Let $E$ be a normed vectorial space or a metric vectorial space, and $G \subseteq \mathcal{L}(E)$ be a topological group of isometries acting on $E$. If $A, B \subseteq E$ are two $G$-invariant subsets, then $A \backslash B$ is again $G$-invariant and $(A \backslash B)_{G} = A_{G} \backslash B_{G}$.
\end{Property}

\section{Equivalence of the Drop theorem, Petal theorem, and Ekeland's variational principle}\label{Equivalència de la gota}

The purpose of this section is to present a proof of the equivalence of the group invariant version of the Ekeland's variational principle (Theorem \ref{Ekeland's altered}), the Petal theorem (Theorem \ref{Petal theorem}), and the Drop theorem (Theorem \ref{Drop theorem}). The proofs presented in this section are modelled on the ideas of the proofs of the original results that can be found in \cite{Penot}. 


Observe that \cite[Theorem 12]{FaIs} is only stated for normed spaces, but an adaptation of the proof shows us the following metric version of the Theorem.
\begin{Theorem}\label{Ekeland's altered}
    Let $(M, d)$ be a complete metric vectorial space and $G \subseteq \mathcal{L}(M)$ be a compact topological group of isometries acting on $M$. Suppose that $d$ is $G$-invariant and convex with respect to $G$, and let $f \colon M \to \mathbb{R}\cup\left\{+\infty\right\}$ be a mapping that is lower semicontinuous, bounded below, proper, $G$-invariant and convex with respect to $G$. Then for all $\gamma > 0$ and every $G$-invariant point $x_{0} \in M$, there exists a $G$-invariant point $a \in M$ such that
    \begin{eqnarray}
    f(a) < f(x) + \gamma d(x, a) & \forall \,  x \in M, \, x \neq a,\label{Uno de Ekeland} \\
    f(a) \leq f(x_{0}) - \gamma d(a, x_{0}). & \label{Dos de Ekeland}
    \end{eqnarray}
\end{Theorem}
\begin{proof}
Since $f$ is proper, it is not identically equal to $+\infty$. If $f(x_{0}) = +\infty$, choose $z \in M$ with $f(z) < +\infty$. By $G$-invariance and convexity with respect to $G$, the point $\overline{z}$ is an invariant base point satisfying $f(\overline{z}) \leq f(z) < +\infty$. The construction below, applied with base point $\overline{z}$, gives an invariant point $a$ with $f(a) < +\infty$ satisfying \eqref{Uno de Ekeland}. Inequality \eqref{Dos de Ekeland} for the original base point $x_{0}$ is then automatic, since its right-hand side is $+\infty$. It therefore suffices to treat the case $f(x_{0}) < +\infty$.

Fix $\gamma > 0$ and a $G$-invariant point $x_{0} \in M$, and consider the set
\[
S = \left\{x \in M ~ | ~ f(x) + \gamma d(x, x_{0}) \leq f(x_{0})\right\}.
\]
Observe that $S$ is a closed subset, since $f$ is lower semicontinuous, that $x_{0} \in S$, and that $S$ is $G$-invariant because so are $f$, $d$ and $x_{0}$. Moreover, $S$ is stable under symmetrization, this is, $\overline{x} \in S$ for every $x \in S$. Indeed, since $f$ is $G$-invariant and convex with respect to $G$, and since $d$ is $G$-invariant, convex with respect to $G$ and $x_{0}$ is $G$-invariant, we have that
\[
\begin{aligned}
f(\overline{x}) &\leq \int_{G}f(g(x))d\mu(g) = f(x), \\
d(\overline{x}, x_{0}) = d(\overline{x}, \overline{x_{0}}) &\leq \int_{G}d(g(x), g(x_{0}))d\mu(g) = d(x, x_{0}),
\end{aligned}
\]
so that $f(\overline{x}) + \gamma d(\overline{x}, x_{0}) \leq f(x) + \gamma d(x, x_{0}) \leq f(x_{0})$. Since $S$ is closed in the complete space $M$, it is complete. For every $g \in G$, the map $x \mapsto d(g(x),x)$ is continuous because $g$ is an isometry. Consequently
\[
S_{G} = S \cap \bigcap_{g \in G}\left\{x \in M ~ | ~ d(g(x),x)=0\right\}
\]
is closed in $M$ and hence complete. Moreover, $f_{|_{S_{G}}}$ is proper because $x_{0} \in S_{G}$ and $f(x_{0}) < +\infty$. Applying the classical Ekeland variational principle \cite[Theorem 1.1]{Ekeland} to $\gamma^{-1}f_{|_{S_{G}}}$, we obtain $a \in S_{G}$ such that $f(a) < f(z) + \gamma d(a,z)$ for every $z \in S_{G}\backslash\left\{a\right\}$. For $x \in S\backslash\left\{a\right\}$ with $\overline{x} \neq a$, apply this inequality to $\overline{x}$ and use $f(\overline{x}) \leq f(x)$ and $d(a,\overline{x}) \leq d(a,x)$. If $\overline{x} = a$, the same strict inequality follows from $f(a) \leq f(x)$, $d(a,x) > 0$ and $\gamma > 0$. Therefore
\[
f(a) < f(x) + \gamma d(a, x) \quad \forall \, x \in S, \, x \neq a.
\]
If $x \in M \backslash S$, then $f(x_{0}) < f(x) + \gamma d(x, x_{0})$, so that, since $a \in S$ and $d$ satisfies the triangle inequality,
\[
f(a) + \gamma d(a, x_{0}) \leq f(x_{0}) < f(x) + \gamma d(x, x_{0}) \leq f(x) + \gamma d(x, a) + \gamma d(a, x_{0}),
\]
and hence the previous inequality also holds for every $x \in M \backslash S$. This shows \eqref{Uno de Ekeland}.

But, since $a \in S$ we have that
\[
f(a) \leq f(x_{0}) - \gamma d(a, x_{0}).
\]
\end{proof}

In the proofs of this section we will apply the previous result not only on the whole space $M$, but also on suitable subsets of $M$ which need not be vector subspaces. For this reason we state separately the following version of Theorem \ref{Ekeland's altered}, in which the symmetrization is required to leave the underlying set invariant. We include its proof here for the sake of completeness.
\begin{Corollary}\label{Ekeland's altered subset}
    Let $(M, d)$ be a metric vectorial space and $G \subseteq \mathcal{L}(M)$ be a compact topological group of isometries acting on $M$. Suppose that $d$ is $G$-invariant and convex with respect to $G$, and let $N \subseteq M$ be a complete $G$-invariant subset which is stable under symmetrization. Let $f \colon N \to \mathbb{R}\cup\left\{+\infty\right\}$ be a mapping that is lower semicontinuous, bounded below, proper, $G$-invariant and convex with respect to $G$. Then for all $\gamma > 0$ and every $G$-invariant point $x_{0} \in N$, there exists a $G$-invariant point $a \in N$ such that
    \begin{eqnarray}
    f(a) < f(x) + \gamma d(x, a) & \forall \,  x \in N, \, x \neq a,\label{Uno de Ekeland restringit} \\
    f(a) \leq f(x_{0}) - \gamma d(a, x_{0}). & \label{Dos de Ekeland restringit}
    \end{eqnarray}
\end{Corollary}
\begin{proof}
As in the proof of Theorem \ref{Ekeland's altered}, we may assume that $f(x_{0}) < +\infty$, replacing an infinite-valued base point by the symmetrization of a finite-valued point of $N$ if necessary. The proof of Theorem \ref{Ekeland's altered} applies with $N$ in the place of $M$. Indeed, given $\gamma > 0$ and a $G$-invariant point $x_{0} \in N$, the set
\[
S = \left\{x \in N ~ | ~ f(x) + \gamma d(x, x_{0}) \leq f(x_{0})\right\}
\]
contains $x_{0}$, is closed in $N$ and hence complete, and is $G$-invariant; moreover, since $\overline{x} \in N$ for every $x \in N$, the two inequalities displayed in the proof of Theorem \ref{Ekeland's altered} show that $\overline{x} \in S$ for every $x \in S$. Thus $S_{G}$ is complete and the fixed-point argument in that proof applies. The argument given there for the points of $M \backslash S$, which only uses the triangle inequality, applies word by word to the points of $N \backslash S$.
\end{proof}

Let us point out that, when $(M, d)$ is complete, this result is a direct consequence of Theorem \ref{Ekeland's altered} applied to the mapping $\widetilde{f} \colon M \to \mathbb{R}\cup\left\{+\infty\right\}$ given by $\widetilde{f}(x) = f(x)$ if $x \in N$, and $\widetilde{f}(x) = +\infty$ otherwise. Indeed, $\widetilde{f}$ is lower semicontinuous because $f$ is lower semicontinuous and $N$ is complete, and hence closed in $M$; it is bounded below and proper; it is $G$-invariant because so are $f$ and $N$; and it is convex with respect to $G$ because $\overline{x} \in N$ for every $x \in N$. Moreover, the $G$-invariant point $a \in M$ given by Theorem \ref{Ekeland's altered} satisfies $\widetilde{f}(a) < +\infty$: otherwise, a finite-valued point of $N$, which exists by properness, would contradict \eqref{Uno de Ekeland}. Thus $a \in N$, and restricting \eqref{Uno de Ekeland} to $N$, together with \eqref{Dos de Ekeland}, gives the conclusion, whether or not $f(x_{0})$ is finite.

\begin{Remark}
The same conclusion holds on a complete metric space $N$ with an isometric group action and a symmetrization map $P \colon N \to N_{G}$ satisfying $P(x)=x$ for $x \in N_{G}$, $d(P(x),P(y)) \leq d(x,y)$, and $f(P(x)) \leq f(x)$, with the same lower semicontinuity, boundedness and properness assumptions on $f$. Indeed, $N_{G}$ is complete and $f_{|_{N_{G}}}$ is proper. Apply the classical Ekeland variational principle on $N_{G}$ and extend its strict inequality to $N$ using $P$, exactly as in the proof of Theorem \ref{Ekeland's altered}.
\end{Remark}

To continue, let us present the $G$-invariant version of the Petal theorem and Drop theorem. 
\begin{Theorem}[$G$-invariant Petal theorem]\label{Petal theorem}
    Let $(M, d)$ be a metric vectorial space, $G \subseteq \mathcal{L}(M)$ be a compact topological group of isometries acting on $M$, and $C \subsetneq M$ a complete $G$-invariant subset which is stable under symmetrization. Assume that $d$ is $G$-invariant and convex with respect to $G$. Let $x_{0} \in C_{G}$, $b \in (M \backslash C)_{G}$, $r \leq d(b, C)$ and $s = d(b, x_{0})$. Then, for all $\gamma > 0$, there exists a $G$-invariant point $a \in C \cap P_{\gamma}(x_{0}, b)$ such that $C \cap P_{\gamma}(a, b) = \left\{a\right\}$.
\end{Theorem}

\begin{Theorem}[$G$-invariant Drop theorem]\label{Drop theorem}
    Let $(X, \Vert \cdot \Vert)$ be a normed space, $G \subseteq \mathcal{L}(X)$ be a compact topological group of isometries acting on $X$, and $C \subset X$ a complete $G$-invariant subset which is stable under symmetrization. Assume that $x_{0} \in C_{G}$, and $B = \overline{B(b, r)}$, where $b \in X_{G}$ and $0 < r < \text{dist}(b, C)$. Then there exists a $G$-invariant point $a \in C \cap D(x_{0}, B)$ such that $C \cap D(a, B) = \left\{a\right\}$.
\end{Theorem}

Let us now present the main goal of this section, that is, to show that the previous three theorems are equivalent.
\begin{proof}[Theorem \ref{Ekeland's altered} $\Rightarrow$ Theorem \ref{Petal theorem}]
Define the function
\[
\begin{array}{cccc}
    f \colon & C & \to & \mathbb{R} \\
     & x & \mapsto & d(x, b),
\end{array}
\]
which is continuous, bounded below by $r$ and $G$-invariant. Note that $f$ is also convex with respect to $G$. Indeed, since $C$ is stable under symmetrization, $\overline{x} \in C$ for every $x \in C$, and since $d$ is convex with respect to $G$ and $b$ is $G$-invariant, we have that
\begin{align*}
f(\overline{x}) &= d(\overline{x}, b) = d(\overline{x}, \overline{b})  \\
& \leq \int_{G}d(g(x), g(b))d\mu(g) \\
&= \int_{G}d(g(x), b)d\mu(g) \\
&= \int_{G}f(g(x))d\mu(g).
\end{align*}
Applying now Corollary \ref{Ekeland's altered subset} to $f$ on $C$, which is a complete $G$-invariant subset of $M$ that is stable under symmetrization, we know that there exists a $G$-invariant point $a \in C_{G}$ such that
\begin{align}
    f(a) < f(x) + \gamma d(a, x) & \quad \forall \, x \in C, \hspace{0.1cm} x \neq a \label{6 de Ekeland} \\
    f(a) \leq f(x_{0}) - \gamma d(a, x_{0}). \label{7 de Ekeland} & 
\end{align}
By \eqref{6 de Ekeland} we know that for every $x \in C\backslash\left\{a\right\}$ we have that $x \notin P_{\gamma}(a, b)$, this meaning that
\[
C \cap P_{\gamma}(a, b) = \left\{a\right\}.
\]
Observe also, that by \eqref{7 de Ekeland},
\[
d(a, b) \leq d(x_{0}, b) - \gamma d(x_{0}, a),
\]
which says precisely that $a \in P_{\gamma}(x_{0}, b)$, and hence
\[
\gamma d(x_{0}, a) \leq -d(a, b) + d(x_{0}, b) = s - d(a, b).
\]
By hypothesis we know that $r \leq d(b, C) \leq d(b, a)$, since $a \in C$, thus
\[
d(x_{0}, a) \leq \frac{s - r}{\gamma}.
\]
This concludes the proof.
\end{proof}

\begin{proof}[Theorem \ref{Petal theorem} $\Rightarrow$ Theorem \ref{Drop theorem}]
Let $B = \overline{B(b, r)}$, and consider $E = C \cap D(x_{0}, B)$. The drop $D(x_{0}, B)$ is closed, convex and $G$-invariant. Observe that $E$ is stable under symmetrization since it is an intersection of two stable sets. Note also that $E$ is complete and $G$-invariant. Define $\rho = d(b, C)$ and $\gamma = \frac{\rho - r}{\rho + r}$. Observe that $\rho > 0$ since, by hypothesis, $d(b, C) > r$ where $r$ is the radius of the closed ball centered at $b$, then $b \in D(x_{0}, B) \backslash C$. By the Petal's theorem there exists a $G$-invariant point $a \in E$ such that $\left\{a\right\} = E \cap P_{\gamma}(a, b)$.

Observe now that, since $t = d(a, b) \geq \rho > r$, it is clear that
\[
\frac{\rho - r}{\rho + r} \leq \frac{t - r}{t + r},
\]
so $D(a, B) \subseteq P_{\gamma}(a, b)$ for $t > r$. Moreover, since $a \in D(x_{0}, B)$, then $D(a, B) \subseteq D(x_{0}, B)$. Therefore
\[
D(a, B) \cap C \subseteq D(a, B) \cap \left(D(x_{0}, B) \cap C\right) \subseteq P_{\gamma}(a, b) \cap E = \left\{a\right\}.
\]
\end{proof}

\begin{Remark}
    In the previous implication, the distance is induced by the norm, so its convexity with respect to $G$ is automatic by Remark \ref{Remark distancia convexa}.
\end{Remark}

In order to prove the last implication we will use the following lemma, whose proof can be found in \cite[Lemma 2.3]{Penot}.

\begin{Lemma}\label{Lema a gota implica Ekeland}
Let $X$ be a normed vector space, $B = \overline{B((0, h), r)} \subseteq X \times \mathbb{R}$ with radius $r \in ]0, h[$ and the norm $\Vert (x, t) \Vert = \max(\Vert x \Vert, |t|)$. Then, the cone $K = \mathbb{R}_{+}B$ generated by $B$ is given by
\[
K = \left\{(x, t) \in X \times \mathbb{R} ~ | ~ t \geq r^{-1}(h - r)\Vert x \Vert\right\}.
\]
\end{Lemma}

Now we can move to the proof of the last implication.
\begin{proof}[Theorem \ref{Drop theorem} $\Rightarrow$ Theorem \ref{Ekeland's altered}]
If $f(x_{0}) = +\infty$, choose $z \in M$ with $f(z) < +\infty$ and replace $x_{0}$ by $\overline{z}$; the second inequality for the original base point is automatic. Thus we may assume that $f(x_{0}) < +\infty$. Put $N = M_{G}$. This is a closed and complete subset of $M$, and $f_{|_{N}}$ is proper. We first prove the two inequalities on $N$.

We start replacing the restricted metric $d$ by $d' = \min(\delta, d)$ where $\delta = \frac{1}{\gamma}(f(x_{0}) - \inf f(N) + 1)$. We may suppose that $d$ is bounded on $N$. Indeed, $d'$ coincides with $d$ for the pairs of points at distance smaller than $\delta$, so that $(N, d')$ is again complete. If $a \in N$ satisfies the two inequalities on $N$ for $d'$, then \eqref{Uno de Ekeland} holds on $N$ for $d$ because $d' \leq d$. Moreover, $d(a, x_{0}) < \delta$, since otherwise $d'(a, x_{0}) = \delta$ and \eqref{Dos de Ekeland} for $d'$ would give $f(a) \leq f(x_{0}) - \gamma\delta = \inf f(N) - 1$, which is impossible. Hence $d'(a, x_{0}) = d(a, x_{0})$ and \eqref{Dos de Ekeland} also holds for $d$. Let $F$ be the normed vector space of the bounded continuous functions on $N$ with the supremum norm. Then, $(N, d)$ can be isometrically embedded in $F$ via the mapping
\[
\begin{array}{ccc}
    (N, d) & \to & F \\
    x & \mapsto & d_{x}(y) = d(x, y).
\end{array}
\]
So we consider $N$ as a complete subset of the normed vector space $F$. Define $E = F \times \mathbb{R}$ with the norm $\Vert (x, t) \Vert = \max\left(\Vert x \Vert, |t|\right)$.

We will apply Theorem \ref{Drop theorem} in $E$ with the trivial group $\left\{Id_{E}\right\}$. Every subset of $E$ is invariant and stable under symmetrization for this group. Thus no compatibility between the above embedding and averaging in $F$ is required.

Without loss of generality we may assume that $x_{0} = 0$ and $f(x_{0}) = 0$. If this were not the case, we could achieve this by replacing the above embedding by $x \mapsto d_{x} - d_{x_{0}}$, which is again an isometric embedding, and $f$ by the mapping $x \mapsto f(x) - f(x_{0})$. We define now $\psi = -f$ on $N$, and observe that $\psi(x_{0}) = f(x_{0}) = 0$. Take $m = \sup\left\{\psi(x) ~ | ~ x \in N\right\}$, $r > \frac{m}{\gamma}$, $h = \gamma r + r > m + r$, and define finally $B = \overline{B((0, h), r)} = \overline{B(0, r)} \times [h - r, h + r]$ and $K = \mathbb{R}_{+}B$. For given $(y_{0}, t_{0}) \in B$ we have that
\[
t_{0} \geq h - r > m,
\]
therefore $(y_{0}, t_{0}) \notin C = \text{Hypo}(\psi) = \left\{(x, t) \in N \times \mathbb{R} ~ | ~ t \leq \psi(x)\right\}$. The set $C$ is complete, since $N$ is complete and $\psi$ is upper semicontinuous. In fact, since $t \leq \psi(x) \leq m$ for every $(x, t) \in C$, we have that $\text{dist}((0, h), C) \geq h - m > r$. Observe also that $(0, 0) \in \text{Hypo}(\psi)$. Hence, by Theorem \ref{Drop theorem} with the trivial group, there exists a point $(a, \alpha) \in C \cap D((0, 0), B)$ such that
\[
\left\{(a, \alpha)\right\} = C \cap D((a, \alpha), B).
\]
Notice that
\[
D((0, 0), B) = \left\{(0, 0) + t(b - (0, 0)) ~ | ~ t \in [0, 1], \, b \in B\right\} =
\]\[
= \left\{tb ~ | ~ t \in [0, 1], \, b \in B\right\} = [0, 1] \cdot B = [0, 1] \cdot \overline{B((0, h), r)}.
\]
Since $(a, \alpha) \in D((0, 0), B)$, in particular $a \in \overline{B(0, r)}$ and $(a, h) \in B$. Then, for every $t_{0} \in [\alpha, h]$, by convexity
\[
(a, t_{0}) = \beta(a, \alpha) + (1 - \beta)(a, h) \quad \hbox{for some } \beta \in [0, 1],
\]
thus $(a, t_{0}) \in D((a, \alpha), B)$. Observe now that it can not happen that $\alpha < \psi(a)$: indeed, since $(a, \alpha) \in C$ we have $\alpha \leq \psi(a) \leq m < h$, so that any $t_{0} \in \, ]\alpha, \psi(a)]$ would satisfy $(a, t_{0}) \in C \cap D((a, \alpha), B)$ with $(a, t_{0}) \neq (a, \alpha)$, a contradiction. Therefore, $\alpha = \psi(a)$.

Applying now Lemma \ref{Lema a gota implica Ekeland} we know that
\[
K = \mathbb{R}_{+}B = \left\{(x, t) \in F \times \mathbb{R} ~ | ~ t \geq r^{-1}(h - r)\Vert x \Vert\right\}.
\]
Since $(a, \alpha) \in K$ and $\alpha = \psi(a)$:
\[
\psi(a) \geq r^{-1}(h - r)\Vert a \Vert = \gamma\Vert a \Vert = \gamma d(a, x_{0}),
\]
from where we deduce \eqref{Dos de Ekeland}, taking into account that $\psi = -f$ and $f(x_{0}) = 0$, i.e.,
\[
f(a) \leq f(x_{0}) - \gamma d(a, x_{0}).
\]
Let now $(x, t) \in (a, \alpha) + K$, where $x \in N\backslash\left\{a\right\}$ and $t \leq m$. It is clear that $t - \alpha \geq \gamma \Vert x - a \Vert > 0$, so we can write
\[
(x - a, t - \alpha) = s(z, h - r - \alpha)
\]
where $z = s^{-1}(x - a)$, $s = \frac{t - \alpha}{h - r - \alpha}$. Observe that, since
\[
t - \alpha \leq m - \alpha < \gamma r - \alpha = h - \alpha - r,
\]\[
h - r - \alpha \geq h - r - m > 0,
\]
it is clear that $s \in \, ]0, 1[$. Moreover, $K$ is a convex cone, so
\[
(a + z, h - r) = (a, \alpha) + (z, h - r - \alpha) \in K,
\]
thus $(a + z, h - r) \in K \cap (F \times \left\{h - r\right\}) \subseteq B$, and by convexity of $D((a, \alpha), B)$ we have that
\[
(x, t) = (a, \alpha) + s((a + z, h - r) - (a, \alpha)) \in D((a, \alpha), B).
\]
In particular $(x, t) \notin C$. For $x \in N\backslash\left\{a\right\}$ with $f(x) < +\infty$, the point $(x,\psi(x))$ belongs to $C$ and satisfies $\psi(x) \leq m$, so $(x, \psi(x)) \notin (a, \alpha) + K$. Therefore
\[
\psi(x) - \psi(a) < \gamma \Vert x - a \Vert.
\]
Thus
\[
f(a) < f(x) + \gamma \Vert x - a \Vert.
\]
This inequality is automatic when $f(x) = +\infty$.

Returning to the original space $M$ and metric $d$, we have obtained $a \in N = M_{G}$ satisfying \eqref{Dos de Ekeland} and the strict inequality on $N\backslash\left\{a\right\}$. For $x \in M\backslash\left\{a\right\}$ with $\overline{x} \neq a$, this gives
\[
f(a) < f(\overline{x}) + \gamma d(\overline{x},a) \leq f(x) + \gamma d(x,a).
\]
If $\overline{x}=a$, then $f(a) \leq f(x)$ and $d(x,a)>0$ give the same strict inequality. Thus \eqref{Uno de Ekeland} holds on all of $M$.
\end{proof}

We would like now to focus on a different version of group invariant Drop theorem and Petal's theorem, where instead of looking for group invariant points in the solution, we search for group invariant sets. Let us give a previous definition.
\begin{Definition}
    Let $(M, d)$ be a metric vectorial space and $G$ be a compact topological group of isometries acting on $M$. For a point $x\in M$ we define 
    \[
    s_G(x):=\inf\{d(x,g(x)) ~ | ~ g \in G \text{ and } g(x) \ne x\},
    \]
    with the convention $\inf\emptyset=+\infty$.
\end{Definition}

The following result is a slight modification of the Petal theorem that allows us to extend the classical Petal theorem to what we call the flower theorem.
\begin{Proposition}
    Let $(M, d)$ be a metric vectorial space, $G \subseteq \mathcal{L}(M)$ be a compact topological group of isometries acting on $M$, and $C \subset M$ a complete $G$-invariant subset of $M$. Let $x_{0} \in C_{G}$, $b \in M \backslash C$. Then, for every $\gamma > 0$, there exists $a \in C \cap P_{\gamma}(x_{0}, b)$ such that
    \[
    C \cap P_{\gamma}(g(a), g(b)) = \left\{g(a)\right\}\text{ for every }g \in G.
    \]
    Furthermore, for every $g,g'\in G$ with $d(g(b),g'(b))>2d(a,b)$ we have that 
    \[
    P_{\gamma}(g(a), g(b))\cap P_{\gamma}(g'(a), g'(b))=\emptyset. 
    \]

\end{Proposition}
\begin{proof}
Apply the original Petal's theorem to obtain $a \in C \cap P_{\gamma}(x_{0},b)$ such that $C \cap P_{\gamma}(a,b)=\left\{a\right\}$. Proposition \ref{Propietat petal} and $g(C)=C$ give the assertion for every $g\in G$. Finally, every $x\in P_{\gamma}(g(a),g(b))$ satisfies $d(x,g(b))\leq d(a,b)$. A point in two such petals would therefore imply $d(g(b),g'(b))\leq 2d(a,b)$ by the triangle inequality, proving the separation assertion.
\end{proof}

\begin{figure}[H]
\begin{center}
\includegraphics[width=.5\textwidth]{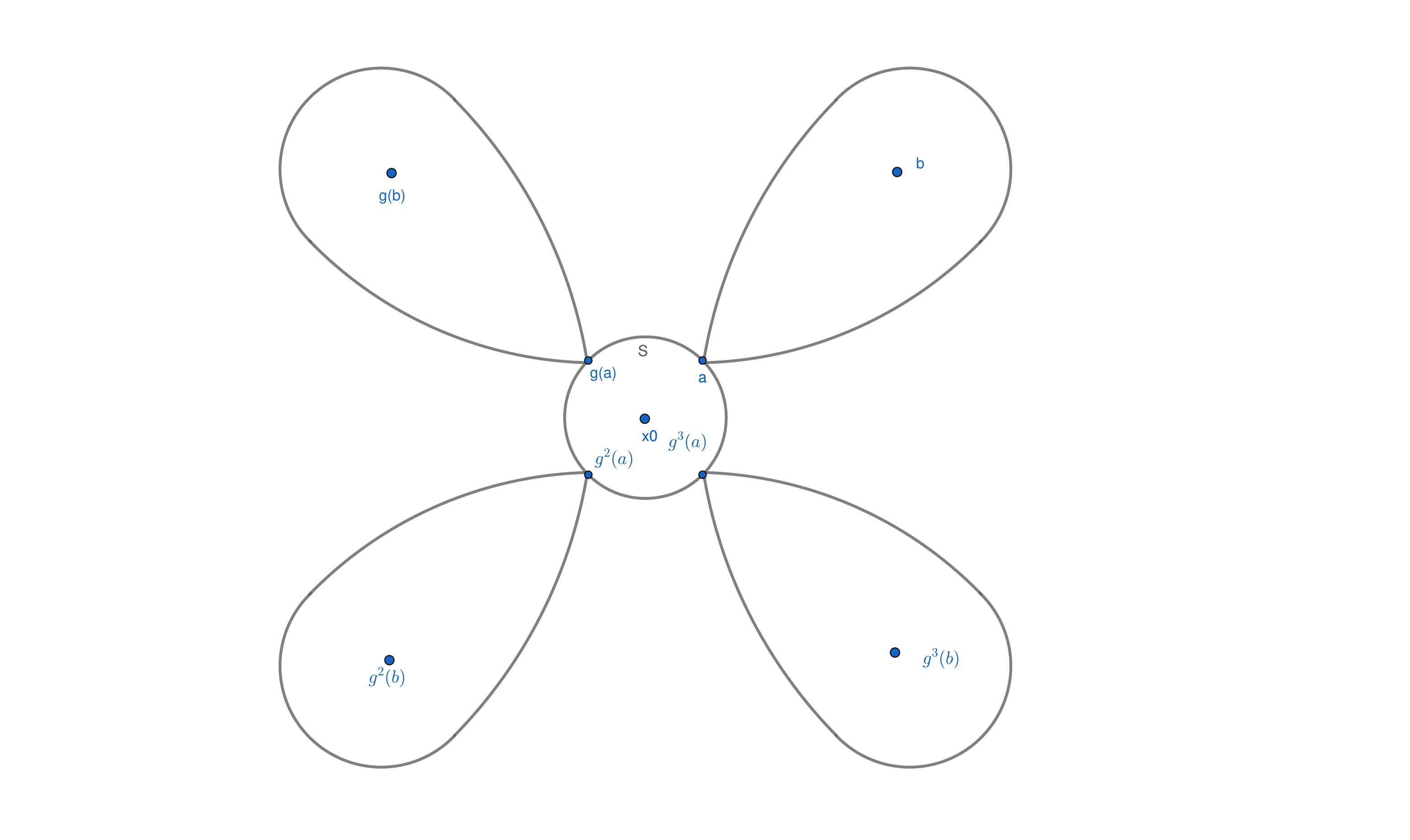}
\end{center}
\caption{Flower theorem statement.}
\label{Figura 5.4}
\end{figure}

We can also obtain an analogous generalization for the Drop theorem.
\begin{Proposition}\label{Generalized drop}
    Let $(X, \Vert \cdot \Vert)$ be a normed space, $G \subseteq \mathcal{L}(X)$ be a compact topological group of isometries acting on $X$, and $C \subset X$ be a complete $G$-invariant subset of $X$. Let $x_{0} \in C_{G}$, $b \in X \backslash C$. Then, there exists $a \in C \cap D(x_{0}, b)$ such that  
    \[
    C \cap D(g(a), g(b)) = \left\{g(a)\right\}\hbox{ for every }g \in G.
    \]
    Furthermore, for every $g,g'\in G$ with $d(g(b),g'(b))>2d(a,b)$ we have that
    \[
    D(g(a), g(b)) \cap D(g'(a), g'(b)) = \emptyset.
    \]
\end{Proposition}
\begin{proof}
Since $C$ is complete, it is closed in $X$, and $C\cap[x_{0},b]$ is nonempty and compact. Choose $a$ in this intersection minimizing the distance to $b$. Every point of $[a,b]\backslash\left\{a\right\}$ is strictly closer to $b$ than $a$, so $C\cap[a,b]=\left\{a\right\}$. Applying $g$ and using $g(C)=C$ gives the first assertion. Each segment $[g(a),g(b)]$ lies in the closed ball centered at $g(b)$ with radius $\Vert a-b\Vert$. The triangle inequality then gives the stated separation condition.
\end{proof}

\begin{figure}[H]
\begin{center}
\includegraphics[width=.7\textwidth]{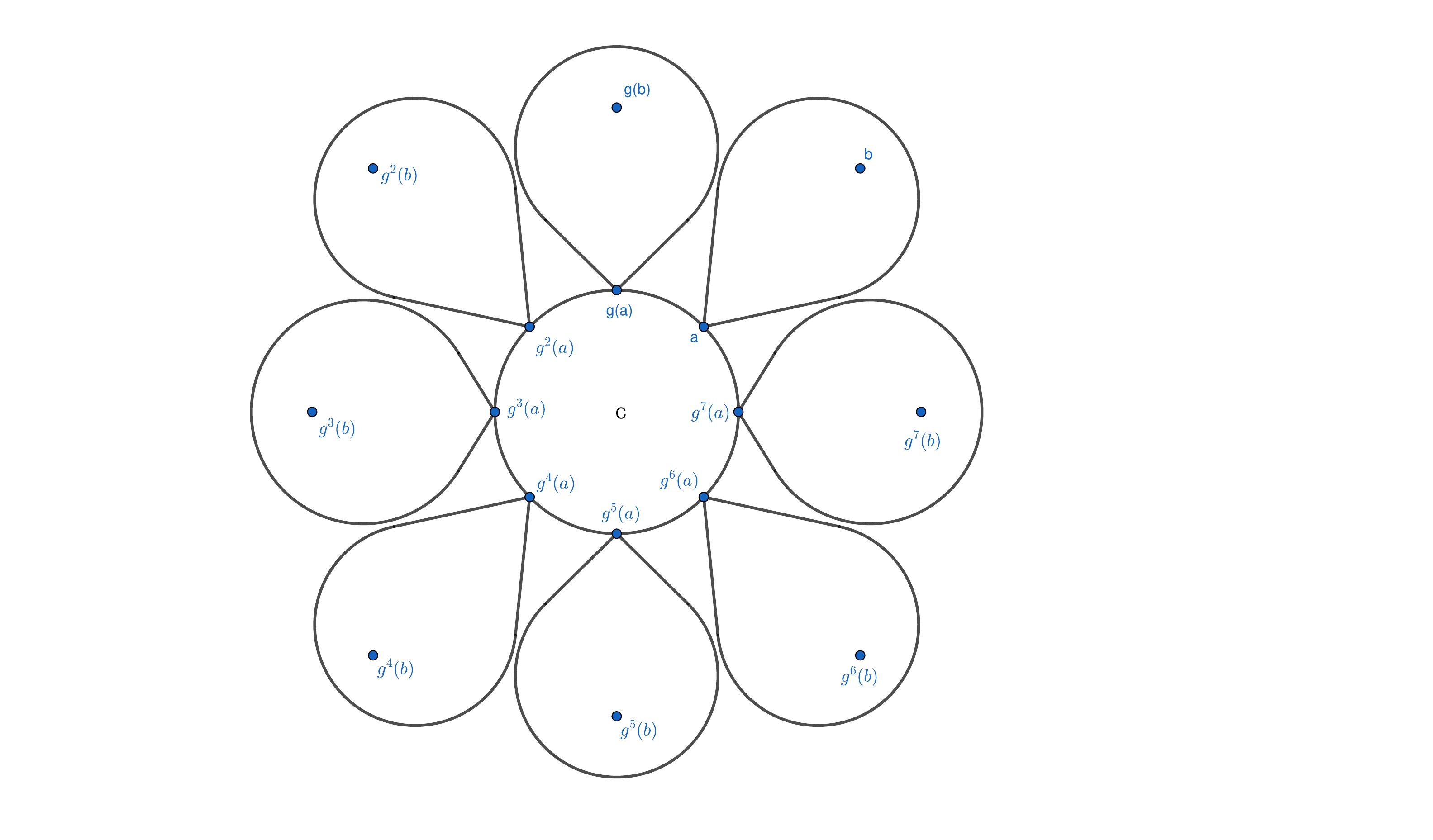}
\end{center}
\caption{Generalized Drop theorem statement.}
\label{Figura 5.5}
\end{figure}

\section{Caristi-Kirk, Ekeland and Takahashi's theorem}

In this section we want to show the equivalence between the Ekeland's variational principle, the Caristi-Kirk fixed point theorem, and the Takahashi's theorem, all of them in the group invariant setting. The ideas of the proofs presented in this section are based on the proofs presented in \cite{OT}. Let us first make some assumptions.

Let $(M, d)$ be a complete metric vectorial space and $G \subseteq \mathcal{L}(M)$ be a compact topological group of isometries acting on $M$, so that $d$ is $G$-invariant. We will also assume that $d$ is convex with respect to $G$ which, by Remark \ref{Remark distancia convexa}, is automatic whenever $d$ is induced by a norm. Let $f \colon M \times M \to \mathbb{R}\cup\left\{+\infty\right\}$ be a function which is lower semicontinuous, $G$-invariant and convex with respect to $G$ in the second coordinate that also satisfies
\begin{equation}\label{condicions primeres equivalències Ekeland}
\begin{aligned}
    f(x, x) &= 0  \quad \forall \, x \in M, \\
    f(x, y) &\leq f(x, z) + f(z, y) \quad \forall \, x, y, z \in M.
\end{aligned}
\end{equation}
Assume that there exists $x_{0} \in M_{G}$ such that
\begin{equation}\label{Condicions segones equivalències Ekeland}
    \inf_{x \in M} f(x_{0}, x) > -\infty,
\end{equation}
and define the set
\begin{equation}\label{Conjunt S0}
S_{0} = \left\{x \in M ~ | ~ f(x_{0}, x) + d(x_{0}, x) \leq 0\right\}.
\end{equation}
Note that by the $G$-invariance of $f$, $d$ and $x_{0}$, the set $S_{0}$ is also $G$-invariant.

Observe moreover that, for every $y \in M$ and every $G$-invariant point $x \in M_{G}$,
\begin{equation}\label{Desigualtats de simetrització}
\begin{aligned}
f(x, \overline{y}) &\leq \int_{G}f(x, g(y))d\mu(g) = f(x, y), \\
d(x, \overline{y}) = d(\overline{x}, \overline{y}) &\leq \int_{G}d(g(x), g(y))d\mu(g) = d(x, y),
\end{aligned}
\end{equation}
since $f$ is convex with respect to $G$ and $G$-invariant in the second coordinate, and $d$ is convex with respect to $G$ and $G$-invariant. In particular, taking $x = x_{0}$, from equation \eqref{Desigualtats de simetrització} we deduce that the set $S_{0}$ is stable under symmetrization.

The assumptions on $f$ and the existence of $x_{0}$ will hold through this section.

In order to show the previously mentioned equivalences we will use the following auxiliary result.
\begin{Theorem}\label{Teorema tècnic}
Let $U \subseteq M$ be a $G$-invariant subset satisfying that
\begin{equation}
    \forall \, y \in S_{0} \backslash U, \, \exists \, x \in M_{G} \hbox{ such that } x \neq \overline{y} \hbox{ and } f(\overline{y}, x) + d(\overline{y}, x) \leq 0.
\end{equation}
Then, there exists $\hat{x} \in (S_{0} \cap U)_{G}$.
\end{Theorem}
\begin{proof}
    Let us construct recursively a sequence of $G$-invariant points $x_{n} \in M$. Consider the initial point $x_{0}$ as the one given in assumption \eqref{Condicions segones equivalències Ekeland}, so
    \[
    \inf_{x \in M}f(x_{0}, x) > -\infty.
    \]
    Given $x_{n}$, define the set
    \[
    S_{n} = \left\{x \in M ~ | ~ f(x_{n}, x) + d(x_{n}, x) \leq 0\right\},
    \]
    and the number
    \[
    \gamma_{n} = \inf_{x \in S_{n}}f(x_{n}, x). 
    \]
    Notice that for $n = 0$, the set $S_{0}$ is the one given in equation \eqref{Conjunt S0}.
    
    Clearly the set $S_{n}$ is $G$-invariant since $x_{n}$ is $G$-invariant, $d$ is $G$-invariant and $f$ is $G$-invariant with respect to the second coordinate. Also by assumption \eqref{condicions primeres equivalències Ekeland} it is clear that $S_{n} \neq \emptyset$, since $x_{n} \in S_{n}$, and $\gamma_{n} \leq 0$.
    
    For $n \geq 1$, suppose $x_{n-1}$ is known and $G$-invariant, and $\gamma_{n-1} > -\infty$. Set a $G$-invariant point $x_{n} \in S_{n-1}$ such that
    \begin{equation}\label{Desigualtat gamma i f}
    f(x_{n-1}, x_{n}) \leq \gamma_{n-1} + \frac{1}{n}.
    \end{equation}
    Such a $G$-invariant point does exist. Indeed, by the definition of $\gamma_{n-1}$ there exists $x \in S_{n-1}$ with $f(x_{n-1}, x) \leq \gamma_{n-1} + \frac{1}{n}$, and its symmetrization $\overline{x}$, which is a $G$-invariant point, satisfies, by \eqref{Desigualtats de simetrització} applied with the $G$-invariant point $x_{n-1}$,
    \[
    f(x_{n-1}, \overline{x}) + d(x_{n-1}, \overline{x}) \leq f(x_{n-1}, x) + d(x_{n-1}, x) \leq 0,
    \]
    so that $\overline{x} \in S_{n-1}$ and $f(x_{n-1}, \overline{x}) \leq f(x_{n-1}, x) \leq \gamma_{n-1} + \frac{1}{n}$; thus it is enough to take $x_{n} = \overline{x}$.
    
    Under this assumptions we are going to show that $S_{n} \subseteq S_{n-1}$. Let $x \in S_{n}$, by assumption \eqref{condicions primeres equivalències Ekeland} and the fact that $x_{n} \in S_{n}$ and $x_{n} \in S_{n-1}$ it is clear that
    \[
    f(x_{n-1}, x) + d(x_{n-1}, x) \leq f(x_{n-1}, x_{n}) + d(x_{n-1}, x_{n}) + f(x_{n}, x) + d(x_{n}, x) \leq 0,
    \]
    so, indeed $x \in S_{n-1}$. Applying now \eqref{condicions primeres equivalències Ekeland} and \eqref{Desigualtat gamma i f} we see that
    \begin{align*}
    \gamma_{n} = \inf_{x \in S_{n}}f(x_{n}, x) &\geq \inf_{x \in S_{n}}\left(f(x_{n-1}, x) - f(x_{n-1}, x_{n})\right)  \\
    &\geq \inf_{x \in S_{n-1}}\left(f(x_{n-1}, x) - f(x_{n-1}, x_{n})\right)  \\
    & = \gamma_{n-1} - f(x_{n-1}, x_{n}) \geq -\frac{1}{n}.
    \end{align*}
    Then, if $x \in S_{n}$
    \[
    d(x_{n}, x) \leq -f(x_{n}, x) \leq -\gamma_{n} \leq \frac{1}{n}.
    \]
    Thus, $\text{diam}(S_{n}) \to 0$. Moreover for every $k \geq n$ it is clear that $x_{k} \in S_{k} \subseteq S_{n}$. In particular,
    \[
    d(x_{k}, x_{n}) \leq \frac{1}{n}.
    \]
    Hence $\left\{x_{n}\right\}$ is a Cauchy sequence of group invariant points. Therefore, since $M_{G}$ is closed, there exists a group invariant point, say $\hat{x}$, which is the limit of the sequence. Observe that every $S_{n}$ is a closed subset of $M$, since $f$ is lower semicontinuous in the second variable and $d$ is continuous, so that $\hat{x} = \lim_{k}x_{k} \in S_{n}$ for every $n \geq 0$, because $x_{k} \in S_{n}$ for every $k \geq n$. Since moreover $\text{diam}(S_{n}) \to 0$, it is clear that
    \[
    \bigcap_{n = 0}^{+\infty}S_{n} = \left\{\hat{x}\right\}.
    \]
    We claim that $\hat{x} \in U$. By contradiction, if $\hat{x} \notin U$, we know by hypothesis that there exists $x \in M$ such that $x \neq \hat{x}$ and
    \[
    f(\hat{x}, x) + d(\hat{x}, x) \leq 0.
    \]
    Also, since $\hat{x} \in \bigcap_{n = 0}^{+\infty}S_{n}$
    \[
    f(x_{n}, \hat{x}) + d(x_{n}, \hat{x}) \leq 0 \quad \forall \, n \geq 0.
    \]
    Now, applying \eqref{condicions primeres equivalències Ekeland}, we obtain
    \[
    f(x_{n}, x) + d(x_{n}, x) \leq 0 \quad \forall \, n \geq 0,
    \]
    this meaning that $x \in \bigcap_{n = 0}^{+\infty}S_{n}$. But this would be a contradiction with the fact that $\bigcap_{n = 0}^{+\infty}S_{n} = \left\{\hat{x}\right\}$. So $\hat{x} \in U$.
\end{proof}

To conclude this section we present the group invariant equivalences of Theorem \ref{Teorema tècnic}, that are the group invariant generalizations of Ekeland's theorem, Takahashi's theorem, and Caristi-Kirk fixed point theorem respectively.
\begin{Theorem}\label{Ekeland}
    Let $(M, d)$ be a complete metric vectorial space and $G \subseteq \mathcal{L}(M)$ be a compact topological group of isometries acting on $M$. Then, the following results are equivalent:
    \begin{enumerate}[(i)]
        \item Let $U \subseteq M$ be a $G$-invariant subset satisfying that
        \begin{equation*}
        \forall \, y \in S_{0} \backslash U, \, \exists \, x \in M_{G} \hbox{ with } x \neq \overline{y} \hbox{ and } f(\overline{y}, x) + d(\overline{y}, x) \leq 0.
        \end{equation*}
        Then, there exists $\hat{x} \in (S_{0} \cap U)_{G}$.
    
        \item There exists $\hat{x} \in S_{0}$ such that $\hat{x}$ is $G$-invariant, and $f(\hat{x}, x) + d(\hat{x}, x) > 0$ for all $x \in M_{G}$, $x \neq \hat{x}$.

        \item Suppose that for all $y \in S_{0}$ satisfying that $\inf_{x \in M}f(\overline{y}, x) < 0$, there exists
        \begin{equation*}
        x \in M_{G} \hbox{ with } x \neq \overline{y} \hbox{ and } f(\overline{y}, x) + d(\overline{y}, x) \leq 0.
        \end{equation*}
        Then, there exists $\hat{x} \in (S_{0})_{G}$ such that $f(\hat{x}, x) \geq 0$ for all $x \in M_{G}$.

        \item Let $T \colon M \rightrightarrows M$ be a multivalued mapping such that for every $y \in S_{0}$ there exists 
        \begin{equation*}
        x \in (T(y))_{G} \hbox{ such that } f(\overline{y}, x) + d(\overline{y}, x) \leq 0.
        \end{equation*}
        Then, there exists $\hat{x} \in (S_{0})_{G}$ such that $\hat{x} \in T(\hat{x})$.
    \end{enumerate}
\end{Theorem}
\begin{proof}
    $(ii)$ $\Rightarrow$ $(i)$
    
    We know that there exists some $G$-invariant point $\hat{x} \in S_{0}$ such that
    \[
    f(\hat{x}, x) + d(\hat{x}, x) > 0 \quad \forall \, x \in M_{G}, \, x \neq \hat{x}.
    \]
    In particular $\hat{x} \in U$: otherwise $\hat{x} \in S_{0} \backslash U$ and, since $\overline{\hat{x}} = \hat{x}$, the hypothesis on $U$ would give a point $x \in M_{G}$ with $x \neq \hat{x}$ and $f(\hat{x}, x) + d(\hat{x}, x) \leq 0$, a contradiction. Hence $\hat{x} \in S_{0} \cap U$.

    $(i)$ $\Rightarrow$ $(ii)$

    Take $y \in M$ and define
    \[
    \Gamma(y) = \left\{x \in M_{G} ~ | ~ x \neq \overline{y}, \hspace{0.2cm} f(\overline{y}, x) + d(\overline{y}, x) \leq 0\right\}.
    \]
    Define now $U = \left\{y \in M ~ | ~ \Gamma(y) = \emptyset\right\}$, which is a $G$-invariant subset of $M$, since $\overline{g(y)} = \overline{y}$ for every $g \in G$, and hence $\Gamma(g(y)) = \Gamma(y)$. Then, if $y \notin U$, by definition, there exists some $x \in M_{G}$ such that $x \in \Gamma(y)$, this is, the hypothesis of $(i)$ is satisfied. Applying now $(i)$ there exists some $G$-invariant point $\hat{x} \in S_{0} \cap U$. So $\Gamma(\hat{x}) = \emptyset$ and, since $\overline{\hat{x}} = \hat{x}$, therefore
    \[
    f(\hat{x}, x) + d(\hat{x}, x) > 0 \quad \forall \, x \in M_{G}, \, x \neq \hat{x}.
    \]

    $(iii)$ $\Rightarrow$ $(i)$

    We proceed by contradiction. Assume that $(S_{0} \cap U)_{G} = \emptyset$, and let $y \in S_{0}$. Since $S_{0}$ is stable under symmetrization, $\overline{y} \in (S_{0})_{G}$, so that $\overline{y} \notin U$, this is, $\overline{y} \in S_{0}\backslash U$. Applying the hypothesis of $(i)$ to $\overline{y}$, and using that $\overline{\overline{y}} = \overline{y}$, there exists some $x \in M_{G}$ such that $x \neq \overline{y}$ and
    \[
    f(\overline{y}, x) + d(\overline{y}, x) \leq 0.
    \]
    From this we can deduce that $\inf\left\{f(\overline{y}, x) ~ | ~ x \in M\right\} < 0$, and that the hypothesis of $(iii)$ is satisfied. Therefore, applying $(iii)$, there exists some $G$-invariant point $\hat{x} \in S_{0}$ such that $f(\hat{x}, x) \geq 0$ for every $x \in M_{G}$. This implies that
    \[
    f(\hat{x}, x) + d(\hat{x}, x) > 0 \quad \forall \, x \in M_{G}, x \neq \hat{x}.
    \]
    But $\hat{x} \in (S_{0})_{G}$, so that $\hat{x} \notin U$ by our assumption, and the hypothesis of $(i)$ applied to $\hat{x}$ provides a point $x \in M_{G}$ with $x \neq \hat{x}$ and $f(\hat{x}, x) + d(\hat{x}, x) \leq 0$. Thus a contradiction, which shows that there exists $\hat{x} \in (S_{0} \cap U)_{G}$.
    
    $(i)$ $\Rightarrow$ $(iii)$
    
    Define
    \[
    U = \left\{y \in M ~ | ~ \inf_{x \in M}f(\overline{y}, x) \geq 0\right\},
    \]
    which is a $G$-invariant subset of $M$, since $\overline{g(y)} = \overline{y}$ for every $g \in G$. Then the hypothesis of $(i)$ follows from the hypothesis of $(iii)$. Since $(i)$ holds, there exists some $G$-invariant point $\hat{x} \in S_{0} \cap U$. And since $\hat{x} \in U$, by definition,
    \[
    \inf_{x \in M}f(\hat{x}, x) \geq 0.
    \]
    In particular $f(\hat{x}, x) \geq 0$ for every $x \in M_{G}$.    

    $(iv)$ $\Rightarrow$ $(i)$

    Define the multivalued map $T \colon M \rightrightarrows M$ as follows
    \[
    T(y) = \left\{x \in M_{G} ~ | ~ x \neq \overline{y}\right\}.
    \]
    Proceed by contradiction, assuming that $(S_{0} \cap U)_{G} = \emptyset$, and let $y \in S_{0}$. As before, $\overline{y} \in (S_{0})_{G}$, so that $\overline{y} \in S_{0}\backslash U$, and the hypothesis of $(i)$ applied to $\overline{y}$ provides a point $x \in M_{G}$ with $x \neq \overline{y}$ and $f(\overline{y}, x) + d(\overline{y}, x) \leq 0$. Since $T(y) \subseteq M_{G}$, we have that $x \in (T(y))_{G}$, so the hypothesis of $(iv)$ follows from $(i)$. Then, there exists some $G$-invariant point $\hat{x} \in S_{0}$ such that $\hat{x} \in T(\hat{x})$. But this is a contradiction with the definition of $T$, since $\overline{\hat{x}} = \hat{x}$.

    $(i)$ $\Rightarrow$ $(iv)$
    
    Define
    \[
    U = \left\{y \in M ~ | ~ \overline{y} \in T(\overline{y})\right\}.
    \]
    This is a $G$-invariant subset of $M$, since $\overline{g(y)} = \overline{y}$. For $y \in S_{0}\backslash U$, we have $\overline{y} \in (S_{0})_{G}$ and $\overline{y} \notin T(\overline{y})$. Applying the hypothesis of $(iv)$ to $\overline{y}$ gives $x \in (T(\overline{y}))_{G}$ with $f(\overline{y},x)+d(\overline{y},x) \leq 0$. Necessarily $x \neq \overline{y}$, so the hypothesis of $(i)$ holds. Hence there exists a $G$-invariant point $\hat{x} \in S_{0}\cap U$. Since $\overline{\hat{x}}=\hat{x}$, the definition of $U$ gives $\hat{x} \in T(\hat{x})$.
\end{proof}

\section{Applications of the Ekeland variational principle}

\subsection{Drop and generalized Drop theorem}

We are going to present a geometrical application of the Drop theorem dealing with the notion of the contingent cone to a subset $C$ of a Banach space $X$. Recall that the contingent cone to $C$ in $a \in C$ is
\[
K_{C}(a) = \limsup_{t \to 0^{+}}t^{-1}(C - a).
\]
So, $v \in K_{C}(a)$ if, and only if, $\liminf_{t \to 0^{+}}t^{-1}d(a+tv, C) = 0$.
\begin{Theorem}
    Let $X$ be a normed space and $G \subseteq \mathcal{L}(X)$ be a compact topological group of isometries acting on $X$. Assume that $C \subseteq X$ is a complete $G$-invariant subset which is stable under symmetrization. Let $x_{0} \in C$ and  $y \in X$ be $G$-invariant and such that the segment $[x_{0}, y]$ is not contained in $C$. Then, for each $\rho > 0$, there exists a $G$-invariant point $a \in C$ such that
    \[
    \Vert x_{0} - a \Vert \leq \Vert x_{0} - y \Vert + \rho, \hbox{ and } y \notin a + K_{C}(a).
    \]
\end{Theorem}
\begin{proof}
    Without loss of generality we may assume that $x_{0} = 0$. Let $\lambda \geq 1$ such that $y = \lambda z$ with $z \in (X \backslash C)_{G}$. Let $r \in ]0, \rho]$ with $r < d(z, C)$ and define $B = \overline{B(z, r)}$. Then, by Theorem \ref{Drop theorem}, we know that there exists a $G$-invariant point $a \in C \cap D(x_{0}, B)$ such that $\left\{a\right\} = C \cap D(a, B)$.

    Let us write now $a = \alpha(z + b)$ with $\alpha \in [0, 1[$ and $\Vert b \Vert \leq r$, and define $t = \frac{\lambda - 1}{\lambda - \alpha}$. It is clear that $t \in [0, 1[$. Define now $w = ta + (1-t)y$. Since $t, \alpha \in [0, 1[$ and $\Vert b \Vert \leq r$, we obtain that
    \[
    \Vert w - z \Vert = \Vert t\alpha b + (t\alpha + (1-t)\lambda - 1)z \Vert = t\alpha \Vert b \Vert < r.
    \]
    Hence $w$ belongs to the interior of $B$, so choose $\eta > 0$ with $B(w,\eta) \subseteq B$. Since $y-a=(1-t)^{-1}(w-a)$, for $0<s<1-t$ the ball
    \[
    B\left(a+s(y-a),\frac{s\eta}{1-t}\right)
    \]
    is contained in $a+\frac{s}{1-t}(B-a)$, which is disjoint from $C$ because $C\cap D(a,B)=\left\{a\right\}$ and $a\notin B$. Consequently
    \[
    \frac{d(a+s(y-a),C)}{s} \geq \frac{\eta}{1-t}>0,
    \]
    and $y-a\notin K_C(a)$. Finally
    \[
    \Vert x_{0} - a \Vert \leq \sup_{u \in B}\Vert x_{0} - u \Vert \leq \Vert x_{0} - z \Vert + r \leq \Vert x_{0} - y \Vert + \rho.
    \]
\end{proof}

\begin{Remark}
    Observe that the condition on $y$ being $G$-invariant cannot be removed within this proof, since we define the point $z$ so that $y = \lambda z$, and we use in this proof that $z$ is $G$-invariant. If $y$ were not $G$-invariant, then the only $G$-invariant point of the line $\left\{\lambda z ~ | ~ \lambda \in \mathbb{R}\right\}$ would be the origin.
\end{Remark}

Applying the preceding theorem with the trivial group and then using the action of $G$, we obtain the following result.
\begin{Theorem}
    Let $X$ be a normed space and $G \subseteq \mathcal{L}(X)$ be a compact topological group of isometries acting on $X$. Assume that $C \subseteq X$ is a complete $G$-invariant subset. Let $x_{0} \in C$ be $G$-invariant, and choose $y \in X$ such that the segment $[x_{0}, y]$ is not contained in $C$. Then, for each $\rho > 0$, there exists $a \in C$ such that, for every $g \in G$,
    \[
    \Vert x_{0} - g(a) \Vert \leq \Vert x_{0} - g(y) \Vert + \rho, \hbox{ and } g(y) \notin g(a) + K_{C}(g(a)).
    \]
\end{Theorem}
\begin{proof}
Apply the preceding theorem with the trivial group to obtain $a \in C$ such that $\Vert x_{0}-a\Vert \leq \Vert x_{0}-y\Vert+\rho$ and $y\notin a+K_C(a)$. Its stability hypothesis is automatic for the trivial group. For every $g\in G$, linearity, continuity and $g(C)=C$ give $K_C(g(a))=g(K_C(a))$. The conclusions follow from this identity, the invertibility of $g$, and $g(x_{0})=x_{0}$.
\end{proof}

\subsection{Partial differential equations}

Finally, we want to give some applications of the group invariant Ekeland's variational principle to the area of partial differential equations. The proofs of this section are generalizations of the proofs presented in \cite{Ekeland}, and \cite{Ekeland2}.

Let us start fixing some notation. During this section $\Omega \subseteq \mathbb{R}^{n}$ is going to be an open-bounded subset of $\mathbb{R}^{n}$ with regular boundary. As usual, $W^{1,1}(\Omega)$ will be the Sobolev space consisting of functions of $L^{1}(\Omega)$ whose first weak derivatives are also in $L^{1}(\Omega)$. $W^{1,1}_{0}(\Omega)$ will be the subset of functions of $W^{1,1}(\Omega)$ that vanish on the boundary of $\Omega$. For $1 \leq p < \infty$, $W^{-1, q}(\Omega)$ will denote the dual of $W_{0}^{1, p}(\Omega)$, where $q$ denotes the conjugate exponent of $p$, i.e., $\frac{1}{p} + \frac{1}{q} = 1$, with the usual convention $q=\infty$ when $p=1$. It is well-known that a function $u \in W^{1, p}(\Omega)$ belongs to $W_{0}^{1, p}(\Omega)$, i.e., $u \in W_{0}^{1, p}(\Omega)$ if, and only if, $u = 0$ on $\partial \Omega$, see \cite[Theorem 9.17]{Brezis}

If $v \in W^{1, 1}(\Omega)$, we will denote its weak derivative with respect to $x_{i}$ by $\frac{\partial v}{\partial x_{i}}$, and $\nabla v$ is going to denote the gradient in the weak sense. If we do not say otherwise, all the derivatives that appear in this section should be understood in the sense of distributions.
\begin{Remark}
    Observe that for a continuously differentiable function, the weak and classical partial derivatives coincide. This is the reason why we denote the weak derivative in the same way as the partial derivative.
\end{Remark}

We will endow $W^{1,p}(\Omega)$ with the norm:
\[
\Vert u \Vert_{1,p} = \Vert u \Vert_{p} + \left\Vert \nabla u\right\Vert_{p}.
\]
In all the applications we are going to study during this section, the group $G$ will be a group of permutations on $\mathbb{R}^{n}$; in particular $G$ is finite. When $G$ acts on scalar functions on a $G$-invariant domain, we write $g^{*}(u)=u\circ g$. Strictly speaking, $g\mapsto g^{*}$ reverses products, since $g^{*}h^{*}=(hg)^{*}$, but the collection $\widetilde G=\{g^{*}~|~g\in G\}$ is itself a finite group of linear isometries. Whenever an abstract variational result is applied to a Sobolev space below, the acting group is this operator group $\widetilde G$; its fixed points are exactly the scalar functions that we call $G$-invariant. Since $G$ is finite, its symmetrization is the usual finite average.

\subsubsection{Plateau problem}
Let us present a result that will be useful in Theorem \ref{Problema de Plateau}, and whose proof can be found in \cite[Corollary 15]{FaIs}.
\begin{Lemma}\label{Lema de la densitat}
Let $X$ be a Banach space and $G \subseteq \mathcal{L}(X)$ be a compact topological group of isometries acting on $X$. Let $f \colon X \to \mathbb{R}$ be a lower semicontinuous function, Gâteaux differentiable, bounded below, $G$-invariant, convex with respect to $G$, and so that there exists constants $k, c > 0$ with
\begin{equation}\label{eq:1}
f(x) \geq k\Vert x \Vert + c \quad \forall \, x \in X.
\end{equation}
Write
\[
X_{G}^{*}=\left\{\ell\in X^{*} ~ | ~ \ell\circ g=\ell \text{ for every }g\in G\right\},
\qquad B_{G}=\left\{\ell\in X_{G}^{*} ~ | ~ \Vert\ell\Vert_{X^{*}}\leq 1\right\}.
\]
Then $\delta f(x)\in X_{G}^{*}$ for every $x\in X_{G}$, and
\[
kB_{G}\subseteq\overline{\left\{\delta f(x) ~ | ~ x\in X_{G}\right\}}^{\Vert\cdot\Vert_{X^{*}}}.
\]
\end{Lemma}
\begin{proof}
For $x\in X_{G}$, differentiating $f(x+t g(h))=f(g(x+th))=f(x+th)$ at $t=0$ gives $\delta f(x)\circ g=\delta f(x)$. If $\ell\in kB_{G}$, then $f-\ell$ is lower semicontinuous, Gâteaux differentiable, $G$-invariant and convex with respect to $G$, and $(f-\ell)(x)\geq(k-\Vert\ell\Vert_{X^{*}})\Vert x\Vert+c\geq c$. Applying Lemma \ref{Palais-Smale minimizing sequences} to $f-\ell$ gives $x_j\in X_G$ with $\Vert\delta f(x_j)-\ell\Vert_{X^{*}}\to0$.
\end{proof}

The notation $X_G^*$ thus denotes a subspace of $X^*$, with its inherited norm. It is isometrically identified with $(X_G)^*$ by restriction: the inverse sends $\lambda\in(X_G)^*$ to $\lambda\circ P$, where $P(x)=\overline{x}$. Indeed, $P$ is a contractive projection onto $X_G$, and every $\ell\in X_G^*$ satisfies $\ell\circ P=\ell$.

The first result of this section guarantees that the perturbed Plateau problem, when perturbing by a $G$-invariant function, has a unique $G$-invariant solution.
\begin{Theorem}\label{Problema de Plateau}
    Suppose $\Omega$, and $v_{0} \in W_{0}^{1,1}(\Omega)$ are $G$-invariant. Then there exists an open neighbourhood $U$ of the origin in $W^{-1,\infty}(\Omega)$ and a subset $\mathcal{T}\subseteq U$ which is dense in $U$ and such that the $G$-invariant elements of $\mathcal{T}$ are dense among the $G$-invariant elements of $U$. For every $G$-invariant $T \in \mathcal{T}$, the perturbed minimal hypersurface equation
    \[
    T = -\sum_{i=1}^{n}\frac{\partial}{\partial x_{i}}\left(\frac{1}{(1 + |\nabla v|^{2})^{\frac{1}{2}}}\frac{\partial v}{\partial x_{i}}\right),
    \]\[
    v - v_{0} \in W_{0}^{1, 1}(\Omega),
    \]
    and the perturbed Plateau's problem
    \[
    \inf \int_{\Omega}\left(1 + |\nabla v|^{2}\right)^{\frac{1}{2}} dx - \langle T, v \rangle,
    \]\[
    v -v_{0} \in W_{0}^{1, 1}(\Omega),
    \]
    both have a unique $G$-invariant solution.
\end{Theorem}
\begin{proof}
Define by
\[
F(v) = \int_{\Omega}\left(1 + |\nabla v + \nabla v_{0}|^{2}\right)^{\frac{1}{2}} dx,
\]
which is the function to be minimized on $W_{0}^{1,1}(\Omega)$. It is known that this function is convex, continuous, and Gâteaux differentiable, with derivative
\[
F'(v) = - \sum_{i=1}^{n}\frac{\partial}{\partial x_{i}}\left(\frac{1}{(1 + |\nabla v + \nabla v_{0}|^{2})^{\frac{1}{2}}}\left(\frac{\partial v}{\partial x_{i}} + \frac{\partial v_{0}}{\partial x_{i}}\right)\right) \in W^{-1, \infty}(\Omega).
\]
Observe that, since $F$ is convex, in particular, is convex with respect to $G$. 

Now we want to check that $F$ is $G$-invariant. Write $g^{*}(v)=v\circ g$. The chain rule gives $\nabla g^{*}(v)(x)=g^{-1}(\nabla v(g(x)))$. Since $g^{*}(v_{0})=v_{0}$ and $g$ preserves the Euclidean norm, the $G$-invariance of $\Omega$ gives
\begin{align*}
F(g^{*}(v)) &= \int_{\Omega}\left(1 + \left|g^{-1}\left(\nabla v(g(x)) + \nabla v_{0}(g(x))\right)\right|^{2}\right)^{\frac{1}{2}}dx \\
&= \int_{\Omega}\left(1 + |\nabla v(g(x)) + \nabla v_{0}(g(x))|^{2}\right)^{\frac{1}{2}}dx\\
&= \int_{\Omega}\left(1 + |\nabla v(x) + \nabla v_{0}(x)|^{2}\right)^{\frac{1}{2}}dx = F(v),
\end{align*}
Finally, we want to see that $F$ satisfies equation \eqref{eq:1} so we can apply Lemma \ref{Lema de la densitat}. Observe that
\begin{equation}\label{Equació Plateau Problem}
\begin{aligned}
\int_{\Omega}|\nabla v|dx - \int_{\Omega}|\nabla v_{0}|dx & \leq \int_{\Omega}|\nabla v + \nabla v_{0}|dx \\
& = \Vert \nabla v + \nabla v_{0} \Vert_{1} \\
& \leq  \int_{\Omega}\left(1 + |\nabla v + \nabla v_{0}|^{2}\right)^{\frac{1}{2}}dx \\
& = F(v),
\end{aligned}
\end{equation}
where we have used the triangle inequality and the inequality $|\xi| \leq (1 + |\xi|^{2})^{\frac{1}{2}}$ respectively. But now, by Poincaré's inequality, see \cite[Corollary 9.19]{Brezis}, we know that there exists a constant $C > 0$ such that
\[
\Vert u \Vert_{1} \leq C \Vert \nabla u \Vert_{1} \quad \forall \, u \in W_{0}^{1,1}(\Omega).
\]
In particular
\[
\Vert \nabla u \Vert_{1} \geq \frac{1}{C+1}\Vert u \Vert_{1,1} \quad \forall \, u \in W_{0}^{1,1}(\Omega).
\]
Therefore, putting this in equation \eqref{Equació Plateau Problem}, we obtain that
\[
F(v) = \int_{\Omega}\left(1 + |\nabla v + \nabla v_{0}|^{2}\right)^{\frac{1}{2}}dx \geq \frac{1}{C+1}\Vert v \Vert_{1,1} - K,
\]
where the constant $K = \Vert \nabla v_{0} \Vert_{1}$. Since moreover $F(v) \geq \int_{\Omega}1 dx = \text{meas}(\Omega) > 0$ for every $v \in W_{0}^{1,1}(\Omega)$, for every $\theta \in \, ]0, 1[$ we obtain that
\[
F(v) = \theta F(v) + (1 - \theta)F(v) \geq \frac{\theta}{C+1}\Vert v \Vert_{1,1} + \text{meas}(\Omega) - \theta\left(\text{meas}(\Omega) + K\right),
\]
so, choosing $\theta < \frac{\text{meas}(\Omega)}{\text{meas}(\Omega) + K}$, the function $F$ satisfies \eqref{eq:1} with the constants $k = \frac{\theta}{C+1} > 0$ and $c = \text{meas}(\Omega) - \theta\left(\text{meas}(\Omega) + K\right) > 0$. Let $Y=W^{-1,\infty}(\Omega)$ and let
\[
U=\{T\in Y~|~\Vert T\Vert_{-1,\infty}<k\},
\qquad
\mathcal{T}=F'\left(W_{0}^{1,1}(\Omega)\right)\cap U.
\]
Applying Lemma \ref{Lema de la densitat} with the trivial group shows that $\mathcal{T}$ is dense in $U$. Apply the same lemma now with the original group acting on $W_{0}^{1,1}(\Omega)$ through the operator group $\widetilde G$. If
\[
Y_{G}=\{T\in Y~|~T\circ g^{*}=T \text{ for every }g\in G\},
\]
then the lemma gives
\[
k\overline{B}_{Y_{G}}\subseteq \overline{F'((W_{0}^{1,1}(\Omega))_{G})}^{\Vert\cdot\Vert_{-1,\infty}},
\qquad
F'((W_{0}^{1,1}(\Omega))_{G})\subseteq Y_{G}.
\]
Consequently $\mathcal{T}\cap Y_{G}$ is dense in $U\cap Y_{G}$: indeed, every $T\in U\cap Y_{G}$ can be approximated by derivatives at invariant points, and all sufficiently close approximants still belong to the open ball $U$. Thus the density assertion holds both in the ambient dual and, relatively, in its invariant subspace. In particular, for every $T \in \mathcal{T}$ the equation $F'(v)=T$ has some solution $v \in W_{0}^{1,1}(\Omega)$. Finally, define for any $T \in W^{-1, \infty}(\Omega)$
\[
F_{T}(v) = F(v) - \langle T, v \rangle.
\]
Then, for any $T \in \mathcal{T}$ there exists $v_{T} \in W_{0}^{1,1}(\Omega)$ such that $F'_{T}(v_{T}) = 0$. Since $F_{T}$ is strictly convex, $v_{T}$ is its unique minimum in $W_{0}^{1,1}(\Omega)$, and convexity identifies its critical points with its minimizers. If $T$ is $G$-invariant, then $F_T$ is $G$-invariant, so uniqueness gives $g^{*}(v_T)=v_T$ for every $g\in G$. Thus $v_T+v_0$ is the unique solution of both stated problems and is $G$-invariant.
\end{proof}

\subsubsection{General partial differential equations}
Let us start with the $G$-invariant version of the Palais-Smale minimizing sequences that will be useful for the proof of Theorem \ref{EDP's}. In the original formulation, the condition of $\varphi$ being lower semicontinuous was used but was missing, we include here the right statement of \cite[Corollary 14]{FaIs} for the sake of completeness.
\begin{Lemma}\label{Palais-Smale minimizing sequences}
Let $X$ be a Banach space and $G \subseteq \mathcal{L}(X)$ be a compact topological group of isometries acting on $X$. Let $\varphi \colon X \to \mathbb{R}$ be lower semicontinuous, Gâteaux differentiable, bounded below, $G$-invariant and convex with respect to the group. Then, there exists a sequence $\left\{x_{n}\right\}_{n=1}^{+\infty} \subseteq X$ such that
\begin{enumerate}
    \item $x_{n}$ is $G$-invariant for all $n \in \mathbb{N}$,
    \item $\varphi(x_{n}) \to \inf\left\{\varphi(x) ~ | ~ x \in X\right\},$
    \item $\Vert \delta\varphi(x_{n})\Vert \to 0$.
\end{enumerate}
\end{Lemma}

Using this lemma, we can obtain the following application.
\begin{Theorem}\label{EDP's}
    Let $\Omega \subseteq \mathbb{R}^{n}$ be a $G$-invariant subset, and let $p \in ]1, +\infty[$. Suppose $f\colon \Omega \times \mathbb{R}^{n} \to \mathbb{R}$ is a borelian function such that
    \begin{enumerate}[(i)]
        \item $f(x, \xi) \geq 0$ $\forall \xi \in \mathbb{R}^{n}$.
        \item $\xi \mapsto f(x, \xi)$ is a convex $C^{1}$ function.
        \item For given constants $a, b \geq 0$, $f$ satisfies a growth condition: $|f'_{\xi}(x, \xi)| \leq a + b|\xi|^{p-1}$ for all $\xi \in \mathbb{R}^{n}$.
        \item $f(g(x), g(\xi)) = f(x, \xi)$ for all $g \in G$, all $x \in \Omega$ and all $\xi \in \mathbb{R}^{n}$.
    \end{enumerate}
    Suppose also that there exists a $G$-invariant function $v_{0} \in W_{0}^{1,p}(\Omega)$ such that $\int_{\Omega}f(x, \nabla v_{0}(x))dx < +\infty$. Then, for all $\epsilon > 0$ there exists a $G$-invariant function $u_{\epsilon} \in W_{0}^{1,p}(\Omega)$ such that
    \begin{enumerate}
        \item $\displaystyle{\left\Vert \sum_{i=1}^{n}\frac{\partial}{\partial x_{i}}\frac{\partial f}{\partial \xi_{i}}(\cdot, \nabla u_{\epsilon}(\cdot)) \right\Vert_{-1,q} \leq \epsilon}$.
        \item $\displaystyle{\int_{\Omega}f(x, \nabla u_{\epsilon}(x))dx \leq \inf\left\{\int_{\Omega}f(x, \nabla v(x))dx ~ \Big| ~ v \in W_{0}^{1,p}(\Omega)\right\} + \epsilon}$.
    \end{enumerate}

\end{Theorem}
\begin{proof}
Define the function
\[
H(u) = \int_{\Omega}f(x, \nabla u(x))dx \quad \forall \, u \in W_{0}^{1,p}(\Omega).
\]
By $(i)$, $H$ is bounded below by $0$ and, arguing as in \cite[Lemma 4.2]{Ekeland}, conditions $(ii)$ and $(iii)$ together with the assumption on $v_{0}$ show that this function is finite everywhere on $W_{0}^{1,p}(\Omega)$, continuous and Gâteaux differentiable, with derivative
\[
H'(u) = -\sum_{i=1}^{n}\frac{\partial}{\partial x_{i}}\left(\frac{\partial f}{\partial \xi_{i}}(\cdot, \nabla u(\cdot))\right) \in W^{-1, q}(\Omega).
\]
Observe now that $H$ is $G$-invariant. Indeed, if $g \in G$ and $g^{*}(u) = u \circ g$, the chain rule gives $\nabla g^{*}(u)(x) = g^{-1}\left(\nabla u(g(x))\right)$, since $g$ is a permutation of the coordinates of $\mathbb{R}^{n}$, so that, by $(iv)$, the change of variables $y = g(x)$ and the $G$-invariance of $\Omega$,
\begin{align*}
H(g^{*}(u)) &= \int_{\Omega}f\left(x, g^{-1}\left(\nabla u(g(x))\right)\right)dx \\
&= \int_{\Omega}f\left(g^{-1}(y), g^{-1}\left(\nabla u(y)\right)\right)dy \\
&= \int_{\Omega}f\left(y, \nabla u(y)\right)dy = H(u),
\end{align*}
for all $u \in W_{0}^{1,p}(\Omega)$, and every $g \in G$. Now, it only remains to check that $H$ is convex with respect to $G$. For given $u \in W_{0}^{1,p}(\Omega)$, observe that
\[
\nabla \overline{u}(x) = \sum_{g \in G}\frac{1}{\vert G\vert}\nabla g^{*}(u)(x),
\]
so that, by the convexity of $\xi \mapsto f(x, \xi)$ given in $(ii)$,
\begin{align*}
H(\overline{u}) = \int_{\Omega}f\left(x, \sum_{g \in G}\frac{1}{\vert G\vert}\nabla g^{*}(u)(x)\right)dx &\leq \int_{\Omega}\sum_{g \in G}\frac{1}{\vert G\vert}f\left(x, \nabla g^{*}(u)(x)\right)dx \\
&= \sum_{g \in G}\frac{1}{\vert G\vert}H(g^{*}(u)) = H(u).
\end{align*}
So we have shown that the inequality $H(\overline{u}) \leq H(u)$ holds, therefore $H$ is convex with respect to the group. Applying now Lemma \ref{Palais-Smale minimizing sequences}, there exists a sequence of $G$-invariant points $\left\{u_{n}\right\} \subseteq W_{0}^{1,p}(\Omega)$ such that 
\begin{equation}\label{eq:2}
H(u_{n}) \to \inf\left\{H(u) ~ | ~ u \in W_{0}^{1,p}(\Omega)\right\}, 
\end{equation}
\begin{equation}\label{eq:3}
    \Vert \delta H(u_{n}) \Vert \to 0.
\end{equation}
Given $\epsilon > 0$, it is enough to take $u_{\epsilon} = u_{n}$ for $n$ large enough: from \eqref{eq:3} one deduces directly $(1)$, and from \eqref{eq:2} one deduces $(2)$.
\end{proof}

\subsubsection{Control theory}

Let $K \subseteq \mathbb{R}^{n}$ be a nonempty $G$-invariant compact convex set, and consider the differential equation
\begin{equation}\label{eq2: Equació diferencial teoria de control}
\begin{cases}
    \frac{dx}{dt}(t)&= f(t, x(t), u(t)),  \\
    x(0)&= x_{0}, 
\end{cases}
\end{equation}
where $x_{0} \in \mathbb{R}^{n}$ and $f \colon [0, T] \times \mathbb{R}^{n} \times K \to \mathbb{R}^{n}$, for fixed $T>0$. Assume that the following properties are satisfied:
\begin{enumerate}[(i)]
    \item $f$ is continuous, $G$-invariant on the 3rd coordinate, and linear with respect to the group on the 3rd coordinate.
    
    \item $\frac{\partial f}{\partial x_{i}}$ is well defined and is continuous for all $1 \leq i \leq n$.

    \item There exists a constant $C > 0$ such that $\langle x, f(t, x, u) \rangle \leq C(1 + |x|^{2})$.
\end{enumerate}
Observe that, since $f$ is $G$-invariant on the 3rd coordinate, the second condition in $(i)$ means that
\[
f\left(t, x, \frac{1}{|G|}\sum_{g \in G}g(u)\right) = \frac{1}{|G|}\sum_{g \in G}f(t, x, g(u)) = f(t, x, u)
\]
for every $t \in [0, T]$, $x \in \mathbb{R}^{n}$ and $u \in K$; thus symmetrizing a control value leaves the vector field unchanged.
We have the following result.
\begin{Theorem}
    Suppose $f$ satisfies the previous assumptions, and let $h \colon \mathbb{R}^{n} \to \mathbb{R}$ be a differentiable function. Then, for all $\epsilon > 0$, there exists a $G$-invariant measurable control $v$, whose trajectory is $y$, such that
    \begin{equation*}
    \left\{
    \begin{array}{c}
        h(y(T)) \leq  \inf h(x(T)) + \epsilon , \\
       \hspace{0.7cm} \langle f(t, y(t), v(t)), p(t) \rangle \leq  \min_{u \in K}\langle f(t, y(t), u), p(t) \rangle + \epsilon,
    \end{array}\right.
    \end{equation*}
    where the infimum is over all measurable controls with values in $K$, the second inequality holds for almost every $t\in[0,T]$, and $p$ is the absolutely continuous solution of the differential system:
    \begin{equation*}
    \begin{cases}
    \frac{dp_{i}}{dt}(t) = -\sum_{j=1}^{n}\frac{\partial f_{j}}{\partial x_{i}}(t, y(t), v(t))p_{j}(t) & \forall 1 \leq i \leq n, \\
    \hspace{0.2cm} p(T) = h'(y(T)). &
    \end{cases}
    \end{equation*}
\end{Theorem}
\begin{proof}
    Define $V$ as the set of all measurable controls $u \colon [0, T] \to K$, identified when they agree almost everywhere. The group acts by $g^{*}(u)(t)=g(u(t))$, so $V$ is $G$-invariant; a fixed point of this action is a control satisfying $g(u(t))=u(t)$ almost everywhere for every $g\in G$. Consider the following metric on $V$,
    \[
    d(u_{1}, u_{2}) = \text{meas}\left\{t \in [0, T] ~ | ~ u_{1}(t) \neq u_{2}(t)\right\}.
    \]
   Then, $V$ is a complete metric space, see \cite[Lemma 10]{Ekeland2}. Observe that $d$ is $G$-invariant, since every $g \in G$ is injective, and that $d$ is convex with respect to $G$, since $u_{1}(t) = u_{2}(t)$ implies $\frac{1}{|G|}\sum_{g \in G}g(u_{1}(t)) = \frac{1}{|G|}\sum_{g \in G}g(u_{2}(t))$. Now, consider the function
    \[
    \begin{array}{cccc}
        F \colon & V & \to & \mathbb{R} \\
         & u & \mapsto & h(x(T)),
    \end{array}
    \]
    where $x$ is the solution of the differential equation \eqref{eq2: Equació diferencial teoria de control}. Observe that, by $(i)$ and $(ii)$, for a measurable control $f(t, x(t), u(t))$ is measurable in time and locally Lipschitz in the state variable, uniformly on compact state sets for $t\in[0,T]$ and $u\in K$. By $(iii)$ its solutions are defined on the whole interval $[0, T]$; hence, for every control $u \in V$, the Cauchy problem \eqref{eq2: Equació diferencial teoria de control} has a unique absolutely continuous solution. It is well-known that $F$ is continuous and bounded below, see \cite[Lemma 7.3]{Ekeland}. Let's see that $F$ is $G$-invariant. Observe that we can express $x$ as follows: 
    \[
    x(t) = x_{0} + \int_{0}^{t}f(s, x(s), u(s))ds.
    \]
    Hence, it is clear that $x$ has some dependence on $u$, to highlight it, we are going to denote it by $x_{u}$. Therefore, by $G$-invariance of $f$ on the 3rd coordinate, the trajectory $x_{u}$ satisfies
    \[
    x_{u}(t) = x_{0} + \int_{0}^{t}f(s, x_{u}(s), u(s))ds = x_{0} + \int_{0}^{t}f(s, x_{u}(s), g(u(s)))ds \quad \forall \, t \in [0, T],
    \]
    this is, $x_{u}$ is also a solution of the Cauchy problem associated with the control $g^{*}(u)$. By the uniqueness of solutions, $x_{g^{*}(u)} = x_{u}$, and hence
    \[
    F(g^{*}(u)) = h(x_{g^{*}(u)}(T)) = h(x_{u}(T)) = F(u).
    \]
    Finally, we have to show that $F$ is convex with respect to $G$. Using the convexity of $K$, the control $\frac{1}{|G|}\sum_{g \in G}g^{*}(u)$ belongs again to $V$ and, since $f$ is linear with respect to the group on the 3rd coordinate,
    \[
    f\left(s, x_{u}(s), \frac{1}{|G|}\sum_{g \in G}g(u(s))\right) = \frac{1}{|G|}\sum_{g \in G}f\left(s, x_{u}(s), g(u(s))\right) = f(s, x_{u}(s), u(s)),
    \]
    so that, arguing as before, $x_{u}$ is a solution of the Cauchy problem associated with the control $\frac{1}{|G|}\sum_{g \in G}g^{*}(u)$. By the uniqueness of solutions, it is clear that
    \[
    F\left(\frac{1}{|G|}\sum_{g \in G}g^{*}(u)\right) = h\left(x_{\frac{1}{|G|}\sum_{g \in G}g^{*}(u)}(T)\right) = h(x_{u}(T)) = F(u).
    \]
    Thus $V$ is stable under symmetrization and $F(\overline{u})=F(u)$. Choose $w\in V$ with $F(w)\leq\inf_V F+\epsilon$ and set $w_0=\overline{w}$. Then $w_0$ is $G$-invariant and $F(w_0)=F(w)$. Applying Corollary \ref{Ekeland's altered subset}, in the metric formulation following its proof, to $F$ on the complete $G$-invariant set $V$, with $\gamma=\epsilon$ and base point $w_0$, we obtain a $G$-invariant point $v\in V$ such that
    \[
    F(v) \leq F(w_0)-\epsilon d(v,w_0) \leq \inf_{V} F + \epsilon,
    \]\[
    F(u) \geq F(v) - \epsilon d(u, v) \quad \forall \, u \in V.
    \]
    We may choose the symmetrized representative of $v$, so that its values are $G$-invariant at every time. From the first inequality it is clear that
    \[
    h(y(T)) \leq \inf h(x(T)) + \epsilon.
    \]
    Taking now $t_{0} \in ]0,T[$ and $k_{0} \in K$, define $u_{\tau} \in V$ for $\tau \geq 0$ as
    \[
    u_{\tau}(t) = k_{0} \quad \hbox{if } t \in [0, T] \cap ]t_{0} - \tau, t_{0}[,
    \]\[
    u_{\tau}(t) = v(t) \quad \hbox{if } t \notin [0, T] \cap ]t_{0} - \tau, t_{0}[.
    \]
    Clearly, $d(u_{\tau}, v) \leq \tau$. Let us denote by $x_{\tau}$ the associated trajectory to $u_{\tau}$, then $u_{0} = v$ and the trajectory associated with $u_{0}$ is $y$, then taking $u_{\tau}$ in the second inequality of Ekeland, we obtain that
    \[
    h(x_{\tau}(T)) - h(y(T)) \geq -\epsilon\tau \quad \forall \, \tau \geq 0.
    \]
    Hence, whenever the right derivative exists,
    \[
    \frac{d}{d\tau}h(x_{\tau}(T))_{|_{\tau=0^{+}}} \geq -\epsilon.
    \]
    For each fixed $k_0$, it is known, see \cite[Theorem 9]{Ekeland2}, that for almost every $t_0$ this right derivative exists and
    \[
    \frac{d}{d\tau}h(x_{\tau}(T))_{|_{\tau=0^{+}}} = \langle f(t_{0}, y(t_{0}), k_{0}) - f(t_{0}, y(t_{0}), v(t_{0})), p(t_{0}) \rangle.
    \]
    Thus
    \[
    \langle f(t_{0}, y(t_{0}), k_{0}) - f(t_{0}, y(t_{0}), v(t_{0})), p(t_{0}) \rangle \geq -\epsilon.
    \]
    First take $k_0$ in a countable dense subset of the compact metrizable set $K$ and intersect the corresponding full-measure sets of times. Continuity in $k_0$ extends the inequality to every $k_0\in K$ at each time in this common full-measure set. Compactness of $K$ then gives the stated minimum.
\end{proof}

\begin{Remark}
The $G$-invariance asserted here concerns the control $v$; no $G$-invariance of the trajectory $y$ is claimed without additional equivariance assumptions in the state variable.
\end{Remark}

\end{document}